\documentclass[12pt]{amsart}
\usepackage{amsmath}
\usepackage{newlfont}
\usepackage{amssymb}
\usepackage{amsfonts}
\usepackage{amsthm}
\usepackage{amscd}
\usepackage[english]{babel}
\usepackage{amsfonts}

\makeatletter
\def\revddots{\mathinner{\mkern1mu\raise\p@
\vbox{\kern7\p@\hbox{.}}\mkern2mu
\raise4\p@\hbox{.}\mkern1mu\raise7\p@\hbox{.}\mkern1mu}}
\makeatother

\theoremstyle{plain}
\newtheorem{thm}{Theorem}[section]
\newtheorem{lem}[thm]{Lemma}
\newtheorem{prop}[thm]{Proposition}

\theoremstyle{definition}

\theoremstyle{remark}

\newcommand{\R}{\mathbb{R}}
\newcommand{\Q}{\mathbb{Q}}
\newcommand{\C}{\mathbb{C}}

\newcommand{\Z}{\mathbb{Z}}
\newcommand{\A}{\mathbb{A}}
\newcommand{\F}{\mathbb{F}}

\newcommand{\Ind}{\mathrm{Ind}}

\newcommand{\wit}{\widetilde}

\title{Degenerate Eisenstein series for $Sp(4)$} 
\author{Marcela Hanzer and Goran Mui\'c}
\address{
Department of Mathematics, 
University of Zagreb,
Bijeni\v cka 30, 10000 Zagreb,
Croatia}
\email{hanmar@math.hr}
\email{gmuic@math.hr}

\date{\today}

\begin{document}

\dedicatory{to Steve Rallis, in memoriam}
\begin{abstract}
In this paper we obtain a complete description of  images and poles of degenerate Eisenstein series 
attached to maximal parabolic  subgroups of $Sp_4(\A)$, where $\A$ is the ring of adeles of $\Q$. 
\end{abstract}


\maketitle
\section{Introduction}

The degenerate Eisenstein series attached to Siegel parabolic subgroups of symplectic and metaplectic  
groups has been studied  extensively for various applications such for example explicit construction of automorphic 
$L$--functions  \cite{Explicit_L-funct}, and application in the Siegel--Weil formula \cite{KR, KR1, KR2, KR3,ICC}. 
On the other hand, we have used a more general type of degenerate Eisenstein series to construct and 
 prove unitarity of various
significant classes of unitary representations of local (real or $p$--adic groups) 
and to construct various families of square--integrable automorphic forms
\cite{Muic1, Muic2, hanzer}. In 
\cite{HM_Eisen_GL} we study degenerate Eisenstein series for $GL_n$, their restriction to archimedean place and images
(this improves in part results of \cite{KR}).

The problem of getting complete information about poles of Eisenstein 
series and their images for classical groups is mostly related to our
 insufficient understanding of various types of degenerate principal 
series representations and standard intertwining  operators  
especially for real groups that appear in the constant term of 
Eisenstein series. For Siegel Eisenstein series this problem is solved in \cite{KR2}.  
In the current literature, there are also some other works which deal with various types of 
''Siegel--like''
degenerate principal series (see for example \cite{HL,L,LL,LL1,gosahi})  but more complicated ones 
appear in the theory of 
automorphic forms (see the last section in \cite{Muic2} for simple examples or \cite{hanzer} for more 
sophisticated examples).

For the group $Sp_4(\mathbb R)$, there is a description  (see \cite{MuicSp4}) of all generalized and 
degenerate principal series in terms of the Langlands classification as well as some  information about the 
images and poles of the local intertwining operators. 
This is used in the present  paper along with the local information on $p$--adic places \cite{STSp4} to get 
the complete description of  the images and poles of degenerate  Eisenstein series for Siegel and 
Heisenberg parabolic subgroups of adelic $Sp_4$.  The description of the residual spectrum and square--integrable non--cuspidal automorphic forms is well--known \cite{Kim}.

Now, we will describe the paper by sections. In Section \ref{prelim} we state notation regarding the 
group 
$Sp_4$, we define degenerate Eisenstein series, and recall basic procedure of  computing their poles 
thorough
the constant term and local normalized intertwining operators. In Section \ref{heisenberg} we deal with the Eisenstein 
series attached to the Heisenberg parabolic subgroup. The main results are Theorems \ref{thm-heisenberg}
and \ref{thm-heisenberg-1}.  In Section \ref{siegel} we deal with the Eisenstein 
series attached to the Siegel parabolic subgroup.  The main results are Theorems \ref{thm:image_Siegel} and 
\ref{thm:image_Siegel_s<0}. As we explain above, this last part has an overlap with the results of \cite{KR2}, but we have 
a different approach to the archimedean components \cite{MuicSp4} which gives us the answer in terms of the Langlands 
classification which is hard to see from \cite{KR2}.

In the future papers we plan to extend this work beyond $Sp_4$ (see also the last Section of \cite{Muic2}). 
One of the obstacles that we need to overcome is better understanding certain degenerate principal series 
which are induced from non--Siegel parabolic subgroups of $Sp_{2n}(\mathbb R)$ (see also the last Section of \cite{Muic2}).

\section{Preliminaries}\label{prelim}
For $n\in \Z_{\ge 1},$ we define $J_n$ as a $n\times n$ matrix with $1'$s on the opposite diagonal, and zeroes everywhere else.
We realize the group $Sp_4$ as a matrix group in the following way:
\[Sp_4(\F)=\left\{g\in GL_4(\F):g^t\begin{bmatrix}
0&J_2\\
-J_2&0
\end{bmatrix}g=\begin{bmatrix}
0&J_2\\
-J_2&0
\end{bmatrix}\right\}.\]
For us $\F\in\{\Q,\Q_p,\R, \A\},$ where $\A$ is the ring of adeles of $\Q.$

The upper triangular matrices in $Sp_4$ form a Borel subgroup $B$, which we fix. The standard parabolic subgroups are those containing this Borel subgroup. The diagonal matrices in the Borel subgroup form a maximal torus, which we denote by $T.$ Thus
\[T(\F)=\{diag (t_1,t_2,t_2^{-1},t_1^{-1});t_1,t_2 \in \F^{\ast}\}.\]
 The unipotent matrices in $B$ form the unipotent radical of $B.$
Let $W$ be the Weyl group of $Sp_4$ with respect to $T.$  We define the action of the Weyl group elements with respect to the elementary reflections: 
\[s(diag (t_1,t_2,t_2^{-1},t_1^{-1}))=diag (t_2,t_1,t_1^{-1},t_2^{-1})\] and 
\[c_2(diag (t_1,t_2,t_2^{-1},t_1^{-1}))=diag (t_1,t_2^{-1},t_2,t_1^{-1}).\]
 All other elements of $W$ are generated by $s$ and $c_2.$
 The set of roots of $Sp_4$ with respect to $T$ is denoted by $\Sigma,$ and $\Sigma^+$ denotes the set of positive roots with respect to the above choice of Borel subgroup. Let $\Delta$ denote the set of simple roots in $\Sigma^{+}.$ Then $\Delta=\{e_1-e_2,2e_2\}$ with obvious meaning of $e_i,\;i=1,2.$

Let $\chi$ denote a unitary G\" rossencharacter of $\Q^\times\backslash \A^\times \longrightarrow \C^\times.$ 
We study the degenerate Eisenstein series on $Sp_4$ acting on the holomorphic sections associated with the global representations of $Sp_4(\A)$ (more precisely, of it's Hecke algebra) induced from the characters of  the maximal standard parabolic subgroups of $Sp_4.$  Thus, we have the Heisenberg and the Sigel case. We denote  by $P_1=M_1U_1$  the Heisenberg parabolic subgroup  of $Sp_4$ so that the standard Levi subgroup $M_1$ is isomorphic to $GL_1\times SL_2$ and in the Siegel case, we denote the Sigel parabolic subgroups $P_2=M_2U_2,$ where now $M_2=GL_2.$ We describe  the corresponding holomorphic sections 
$f_s$ in each of these cases more thoroughly in the second and the third section of this paper. In both cases, we form the degenerate Eisenstein series
\begin{equation}\label{de-1}
E(f_s)(g)\overset{def}{=}\sum_{\gamma \in P_{i}(\Q)\setminus Sp_4 (\Q)}
f_s(\gamma \cdot g)
\end{equation}
  which converges absolutely and uniformly  in $(s, g)$ on compact sets when $s>2$ in the Heisenberg case, and $s>3$ in the Siegel case. This
 is proved by the restriction to $Sp_4(\R)$  and then applying
 Godement's theorem as in (\cite{Borel1966}, 11.1 Lemma). In particular, there are no poles for such $s.$

It continues 
to a function which is meromorphic in $s$. Outside of poles, it 
is an automorphic form. As usual and more convenient for computations,
we write $E(s,f)$ instead of $E(f_s)$; in this notation $s$ signals
that $f\in I(s)$.

We say that  $s_0 \in \mathbb C$ is a pole of the degenerate Eisenstein
series $E(s, \cdot)$ if there exists $f\in I(s)$ such that $E(s, f)$
has a pole at $s=s_0$ (for some choice of $g\in Sp_4(\A)$). The order of pole at $s_0$ is denoted by $l$; it is supremum 
of all orders $E(s, f)$ at $s_0$ when $f$ ranges over $I(s)$. It may happen that $l=\infty$ as it can be seen from our 
main theorems but if $0\le l< \infty$, then the map 
\begin{multline}\label{de-100}
\begin{CD}
\Ind_{P_{i}(\A)}^{Sp_4(\A)}(\pi(\chi)_i)  @>f\mapsto  (s-s_0)^lE(s, f)>>
\cal A\left(Sp_4(\Q)\setminus Sp_4(\A)\right)
\end{CD}
\end{multline}
for $i=1,2$ is an intertwining operator for the action of 
$\left(\mathfrak{sp}(4), K_\infty\right)\times
\prod_{p<\infty}Sp_4(\Q_p)$ in the space of automorphic forms. Here 
\begin{equation}
\label{eq:defnpi}
\pi(\chi)_1=\chi|\det|_{GL_1(\A)}^{s_0}\otimes 1_{SL_2(\A)} \text{ and } \pi(\chi)_2=\chi|\det|_{GL_2(\A)}^{s_0}.
\end{equation}

The poles of the Eisenstein series are the same as the poles of its constant term along 
the minimal parabolic subgroup:
\begin{equation}\label{de-2}
E_{const}(s,f)(g)= \int_{U(\Bbb Q)\setminus U( \Bbb A)} E(s,f)(ug)du.
\end{equation}
Here $U$ denotes the unipotent radical of the (upper triangular) Borel subgroup of $Sp_4.$
The integral in (\ref{de-2}) can be computed by the standard unfolding.
To explain this we introduce some more notation. We use $\nu$ to denote $|\det|$ on $GL_n$ simultaneously. The rank of the general linear group involved will be obvious from the context. We let 
$\Lambda_s=\chi \nu^s\otimes \nu^{-1}$ if we are in the Heisenberg case and $\Lambda_s=\chi\nu^{s-1/2}\otimes \chi\nu^{s+1/2}$ if we are in the Siegel case. 

In this way, we obtain a character of $T(\Q)\setminus T(\A)\rightarrow \C^\times.$ 
We extended trivially across $U(\A)$ and we induce up to the principal series $\Ind_{T(\A)U(\A)}^{Sp_4(\A)}(\Lambda_s).$

We denote by $\overline{U}$ the lower unipotent triangular matrices in $Sp_4$ (these matrices form the unipotent radical of the opposite Borel subgroup in $Sp_4.$)
Let $w\in W$. Then, the global intertwining operator
$$
M(\Lambda_s, w): \Ind_{T(\A)U(\A)}^{Sp_4(\A)}(\Lambda_s)\longrightarrow 
 \Ind_{T(\A)U(\A)}^{Sp_4(\A)}(w(\Lambda_s)),
$$
defined by 
$$
M(\Lambda_s, w)
f =\int_{U(\A)\cap
w\overline{U}(\A)w^{-1}}f(\wit{w}^{-1}ug)du
$$
does not depend on the choice of the representative   for
$w$ in $Sp_4(\Q)$. 

The global intertwining operator
factors into product of local intertwining operators 
$$
M(\Lambda_s, w)f=\otimes_{p\le \infty} A(\Lambda_{s, p} ,\wit{w})f_p.
$$
There is a precise way of normalization of Haar measures used in the 
definition of intertwining operators 
\cite{Sh1}, \cite{Sh2}. Summary can be found in (\cite{Muic1}, Section 2) or 
(\cite{Muic2}, Section 2). The same is with the normalization factor which we explain 
next. The normalization factor for $A(\Lambda_{s, p}, \wit{w})$ is defined by

$$
r(\Lambda_{s, p} , w)=
\prod_{\alpha\in \Sigma_+,  w(\alpha)<0}\frac{ 
L(1, \Lambda_{s, p} \circ\alpha^\vee)\epsilon(1, \Lambda_{s, p}  \circ\alpha^\vee, \psi_v)}{
L(0, \Lambda_{s, p} \circ\alpha^\vee)},$$
where $\alpha^\vee$ denotes the coroot corresponding to the root
$\alpha,$ and $\psi_v$ is an non-degenerate additive character of $\Bbb
Q_p.$
We define the normalized intertwining operator by the following formula:
$$
\cal N(\Lambda_{s, p} , \wit{w})=r(\Lambda_{s, p} , w)
A(\Lambda_{s, p} , \wit{w}).
$$
 Properties of normalized intertwining operators can be found in \cite{Sh1}, \cite{Sh2}.
Again, the summary can be found in (\cite{Muic2}, Theorem 2-5). 

Let us write $\beta$ for the simple root such that $\Delta-\{\beta\}$
determines $P_i$.  Now, the constant term has the 
following expression (\cite{Muic1}, Lemma 2.1):
\begin{multline*}
E_{const}(s, f)(g)=\sum_{w\in W, \  w(\Delta\setminus\{\beta\})>0} 
M(\Lambda_s, w)f(g)\\
=\sum_{w\in W, \  w(\Delta\setminus\{\beta\})>0} 
\int_{U(\Bbb A)\cap w\overline{U}(\Bbb A)w^{-1}}f(\wit{w}^{-1}ug)du,
\end{multline*}
where, by induction in stages, we identify
\begin{equation}\label{de-30000000000000}
f\in \Ind_{P_{i}(\A)}^{Sp_4(\A)}\left(\pi(\chi)_i\right)
\subset 
\Ind_{T(\A)U(\A)}^{Sp_4(\A)}(\Lambda_s)
\end{equation}
where $\pi(\chi)_i,\;i=1,2$ is defined in (\ref{eq:defnpi}).
This formula can be more refined up to its final form that we use. Let $S$ be the 
finite set of all places including $\infty$ such that for $p\not\in S$ we have that
$\chi_p$, $\mu_p$, $\psi_p$, and $f_p$ are unramified. Then
 we have the following expression:
\begin{equation}\label{de-3}
E_{const}(s, f)(g)= \sum_{w\in W, \  w(\Delta\setminus\{\beta\})>0} r(\Lambda_{s} , w)^{-1}
\left(\otimes_{p\in S} \cal N(\Lambda_{s, p} , \wit{w})f_p\right)\otimes 
\left(\otimes_{p\not\in S} f_{w, p}\right),
\end{equation}
where we let
\begin{equation}\label{de-5}
r(\Lambda_{s} , w)^{-1}\overset{def}{=}\prod_{\alpha\in \Sigma^+,  w(\alpha)<0}\frac{
L(0, \Lambda_{s} \circ\alpha^\vee)}{ 
L(1, \Lambda_{s} \circ\alpha^\vee)\epsilon(1, \Lambda_{s}  \circ\alpha^\vee)},
\end{equation}
and we use a well--known property of normalization
\begin{equation}
\label{de-5000000000}
\cal N(\Lambda_{s, p} , \wit{w})f_p =f_{w, p},
\end{equation}
for unramified $f_p$ and $f_{w, p}$.

We note that, in the situation as above, taking the constant term is
an isomorphism between different spaces of automorphic forms (well
known fact, e.g. \cite{HM_Eisen_GL},  Lemma 2-9.)

We use standard notation \cite{STSp4} (see also\cite{MuicSp4})  for the representation theory of 
classical groups (in local or global  settings). In more detail, if $\chi$ is a character of $GL(1)$ and $\pi$ is a 
representation of $SL_2$, then $\chi\rtimes \pi$ denotes the representation unitarily induced from $P_1$ to $Sp_4$. 
If $\pi$ is a representation of $GL(2)$, then $\pi\rtimes 1$ denotes the representation unitarily induced from $P_1$ 
to $Sp_4$. Also, if $\chi$ and $\mu$ are characters of $GL(1)$, then  $\chi\times \mu\rtimes 1$ is the associated principal series of $Sp_4$. Similar notation  is used for $GL(2)$ and $SL_2$. We denote by $L(\ \ )$ the Langlands quotient
 whenever in parenthesis is an induced representation having a Langlands quotient.

We use repeatedly the following simple fact (and similarly for $GL(2)$):

\begin{lem} \label{norm-pol} Let $p\le \infty$. The complex number $s=s_0\in \C$ is  a pole of 
$\cal N(s, \mu_p,  w)$ (a normalized intertwining operator) if and only if  $Re(s_0)<0$ and 
$\Ind_{P(\Q_p)}^{SL_2(\Q_p)}(|\
  |^{s_0}_p\mu_p)$ is reducible.
\end{lem} 
\begin{proof} Indeed, the assumption $Re(s_0)<0$ is clear since we known that 
$\cal N(s, \mu_p,  w)$ is holomorphic and non--trivial for $Re(s)\ge 0$. Assume $Re(s_0)<0$. Then, 
if $\Ind_{P(\Q_p)}^{SL_2(\Q_p)}(|\
  |^{s_0}_p\mu_p)$ is irreducible, then the functional equation 
$$\cal N(s, \mu_p,
  w)\cal N(-s, \mu^{-1}_p,  w^{-1})=\cal N(-s, \mu^{-1}_p,  w^{-1})\cal N(s, \mu_p,  w)=id,$$ 
 combined with holomorphy and non--triviality of  $\cal N(-s_0,
  \mu^{-1}_p,  w^{-1})$ imply that $\cal N(s, \mu_p,
  w)$ is holomorphic at $s=s_0$. Conversely,  still assuming
  $Re(s_0)<0$,   if  $\cal N(s, \mu_p,   w)$ is holomorphic for $s=s_0$, then
  then the functional equation, combined with holomorphy and non--triviality of  $\cal N(-s_0,
  \mu^{-1}_p,  w^{-1})$, imply  
$$
\Ind_{P(\Q_p)}^{SL_2(\Q_p)}(|\  |^{-s_0}_p\mu_p)\simeq 
\Ind_{P(\Q_p)}^{SL_2(\Q_p)}(|\  |^{s_0}_p\mu_p).
$$
Then the argument with the Langlands quotient implies that 
$\Ind_{P(\Q_p)}^{SL_2(\Q_p)}(|\  |^{s_0}_p\mu_p)$ is irreducible.
\end{proof}

\section{The Heisenberg parabolic}\label{heisenberg}
The standard Levi subgroup of the Heisenberg parabolic subgroup $P_1$  is
isomorphic to $GL_1\times SL_2,$ thus we study the global induced
representation 
\[\Ind_{GL_1(\A)\times SL_2(\A)}(\chi \nu^{s}\otimes 1),\]
 where
$\chi$ denotes a Grossencharacter of $\Q$ and and $1$ is a trival
character of  $SL_2(\A).$ This space consists of   all $C^\infty$ and right $K$--finite functions 
$f: Sp_{4}(\A)\longrightarrow \C$
which satisfy
\begin{align*}
&f\left(\left(\begin{matrix}a& 0&0\\ 0 & d&0\\
0&0&a^{-1}\end{matrix}\right)
\left(\begin{matrix}1& b_1&b_2\\ 0 & I_2&b_3\\0&0&1\end{matrix}\right) g\right)=
\chi(a)\nu(a)^{s}\delta_{P_{1}}^{1/2}\left(\left(\begin{matrix}a& 0&0\\ 0 & d&0\\
0&0&a^{-1}\end{matrix}\right)\right)f(g), \\
&\text{where} \ \ 
a\in GL_1(\A),\;d\in SL_2(\A),\;\left(\begin{matrix}1& b_1&b_2\\ 0 & I_2&b_3\\0&0&1\end{matrix} \right)\in U_1(\A),\
g\in  Sp_4(\A).
\end{align*}

So, for the case of the Heisenberg maximal parabolic, a simple root
$\beta$ such that $\Delta \setminus \{\beta\}$ determines Heisenberg
parabolic is $\beta=e_1-e_2;$ i.e.,  in the expression for the
constant term (\ref{de-3}) we have a summation over $w\in W;
w(2e_2)>0.$ We easily check that these elements are $\{1,
c_1,s,sc_1\}.$
We note that using $s$ and $c_2$ (as the Weyl reflections with respect
to the simple roots of $Sp_4$ we can express $c_1$ and $sc_1$ as
$sc_2s$ and $c_2s,$ respectively, so that $c_1$ is of length  three,
and $sc_1$ is of length two. Now we easily see that the expression for
the (global) normalizing factors (obtained as the product of the local
ones) for the intertwining operators attached to
these Weyl group elements are 

\begin{equation}
\label{rc_1}
r(\Lambda_s,c_1)^{-1}=\frac{L(s-1,\chi)}{L(s+2,\chi)\varepsilon(s+2,\chi) \varepsilon(s,\chi) \varepsilon(s+1,\chi)},
\end{equation}
\begin{equation}
\label{rs}
r(\Lambda_s,s)^{-1}=\frac{L(s+1,\chi)}{L(s+2,\chi)\varepsilon(s+2,\chi)},
\end{equation}
\begin{equation}
\label{rsc_1}
r(\Lambda_s,sc_1)^{-1}=\frac{L(s,\chi)}{L(s+2,\chi)\varepsilon(s+1,\chi) \varepsilon(s+2,\chi)} .
\end{equation}

We  want to address the holomorphicity of these expressions. We immediately see the following:
\begin{lem}
\label{lem:normalization}
\begin{enumerate}
\item Assume $s\ge 0.$ Then,  the denominators of all three of the
  above expressions are non-zero and holomorphic. Thus, the poles
  cannot come from the zeroes of the denominators.  As for the
  numerators, they  are all holomorphic if $\chi \neq 1,$ and if
  $\chi=1$ we have the following:
\begin{enumerate}
\item  $r(\Lambda_s,c_1)^{-1}$ has a pole of the first order for $s=1$
  and $s=2,$
\item $r(\Lambda_s,s)^{-1}$ has a pole of the first order for $s=0,$
\item $r(\Lambda_s,sc_1)^{-1}$ has a pole of the first order for $s=0$
  and $s=1.$
\end{enumerate}

\item Assume $s<0.$  We analyze the zeroes of the denominators. The
  denominators (for all three expressions)  might have zeroes in the critical strip, i.e.,
  $0<s+2<1,$ i.e., $-2<s<-1.$ The numerators do not have poles if
  $\chi \neq 1$ and do have poles for $r(\Lambda_s,s)^{-1}$ for $s=-1$
  and $\chi=1.$
\end{enumerate}
\end{lem}
Now we analyze local intertwining operators appearing in (\ref{de-3}).

\subsection{Local intertwining operators appearing in  (\ref{de-3})}
We recall  that for every $p,$ the trivial representation of $SL_2(\Q_p)$ is
embedded in the principal series representation $\nu_p^{-1}\rtimes 1.$
\begin{lem}
\label{lem:c_1nonarch}
The local intertwining operator $ \cal N(\Lambda_{s, p} , \wit{c_1})$
acting on $\chi_p\nu^{s}\rtimes 1_p$ (where $1_p$ is the trivial
representation of $SL_2(\Q_p)$)  is holomorphic for every $p<\infty$ and
 for every $s\in \R,$ except for $s=-2,$ where it has a pole of the
 first order (for every $p<\infty$).
\end{lem}
\begin{proof}
According to the decomposition  $c_1=sc_2s$ we have the following
decomposition of the local intertwining operator  $ \cal N(\Lambda_{s,
  p} , \wit{c_1}):$
\begin{multline}
\label{decomp_c1}
 \chi_p\nu^{s}\times \nu_p^{-1}\rtimes 1\to  \nu_p^{-1}\times
 \chi_p\nu_p^{s}\rtimes 1 \to  \nu_p^{-1}\times \chi_p^{-1}\nu_p^{-s}\rtimes 1 \\
\to \chi_p^{-1}\nu_p^{-s}\times \nu_p^{-1}\rtimes 1.
\end{multline}
Then, the first (normalized) intertwining
operator appearing in the above relation is induced from the
$GL_2$-case and is holomorphic unless $\chi_p=1$ and $s=-2.$
The second intertwining operator  is induced from  the intertwining
operator $\chi_p\nu_p^{s}\rtimes 1\to \chi_p^{-1}\nu_p^{-s}\rtimes 1
.$
This intertwining operator is holomorphic if $s>0,$ (from the
Langlands' condition), and for $s\le 0,$ this operator is holomorphic
unless the induced representation  $\chi_p\nu_p^{s}\rtimes 1$ is
reducible, and this happens if $\chi_p^2=1,\,\chi_p \neq 1$ and $s=0$
and if $\chi_p=1$ and $s=-1.$ If we examine the first case more
closely, we see that then the unnormalized intertwining operator is
holomorphic since the Plancherel measure does not have a zero for
$s=0$ in that case, and the normalizing factor is holomorphic, too, so
that we actually have holomorphicity. On the other hand,  if
$\chi_p=1$ and $s=-1$ the pole of the normalized intertwining
operator occurs for the unique quotient of the representation
$\nu_p^{-1}\rtimes 1$ and this is the Steinberg representation of
$SL_2(\Q_p).$ The third intertwining operator is holomorphic unless
$\chi_p=1$ and $s=0.$

\noindent Now we examine the case $\chi_p=1$ and $s=0.$
Note that in this case 
\begin{equation}
\label{eq:s=0}
\cal N(\Lambda_{s, p} , \wit{c_1}):\nu_p^0\times \nu_p^{-1}\rtimes 1 \to
\nu_p^0\times \nu_p^{-1}\rtimes 1 .
\end{equation}
It is known that this induced representation is of length four (\cite{STSp4}, Proposition 5.4 (ii))) and
that $\nu_p^0\rtimes 1 $ is in irreducible tempered representation of
$SL_2(\Q_p).$
Then, $ \nu_p^0 \rtimes 1_{SL_2(\Q_p)}=L(\nu^1;\nu_p^0\rtimes 1)\oplus
L(\nu^{1/2}St_{GL_2(\Q_p)};1).$ The (normalized) spherical vector belongs
to (and generates)  representation $L(\nu^1;\nu_p^0\rtimes 1),$ and according to
(\ref{de-5000000000}), $\cal N(\Lambda_{s, p} , \wit{c_1})$ acts on it
as the identity. Analogously, for $s=0$  $\cal N(\wit{c_1}(\Lambda_{s, p}),
\wit{c_1})$ acts as the identity on the spherical vector. We know that $\cal N(\wit{c_1}(\Lambda_{s, p}) ,
\wit{c_1})\cal N(\Lambda_{s, p} , \wit{c_1})=Id.$ Assume that  $\cal N(\wit{c_1}(\Lambda_{s, p}) ,
\wit{c_1})$ has a pole of order $n_2$ for $s=0,$ and that $\cal
N(\Lambda_{s, p} , \wit{c_1})$ has  a pole of order $n_1$ for $s=0$ on
$L(\nu^{1/2}St_{GL_2(\Q_p)})$ (which is a subrepresentation of
$\nu_p^0\times \nu_p^{-1}\rtimes 1$ and appears there with the
multiplicity one). 
Let $N_1=\lim_{s \to 0}s^{n_1}\cal N(\Lambda_{s, p} , \wit{c_1})$ and
$N_2=\lim_{s \to 0}s^{n_2}\cal N(\wit{c_1} (\Lambda_{s, p}) ,
\wit{c_1}) .$ This means that $N_1$ is a holomorphic isomorphism on
$L(\nu^{1/2}St_{GL_2(\Q_p)}),$ and so is $N_2.$ But if $n_1>0$ or
$n_2>0$ the composition $N_2N_1\vert_{L(\nu^{1/2}St_{GL_2(\Q_p)})}
=0,$ which is thus impossible. This means that $\cal N(\Lambda_{s, p}
, \wit{c_1})$ is holomorphic on $ \nu_p^0 \rtimes 1_{SL_2(\Q_p)}$ and
the dual of the commuting algebra theorem (\cite{Ban_Aubert}) says that on
$L(\nu^{1/2}St_{GL_2(\Q_p)};1)$ it acts as $-Id.$

\noindent We know examine the case $\chi_p=1$ and $s=-1.$
In this case, both intertwining operators corresponding to $s$ (in the
decomposition above corresponding to the decomposition $c_1=sc_2S$)
are holomorphic isomporphisms, and the second operator has a pole on
the representation $\nu^{-1}\rtimes St_{SL_2(\Q_p)}$ (which is an
irreducible quotient of $\nu^{-1}\times \nu^{-1}\rtimes 1$) but it is holomorphic
on $\nu_p^{-1}\rtimes 1_{SL_2(\Q_p)}$ (irreducible by Proposition 5.4
(ii) of \cite{STSp4}).  The image $\cal N(\Lambda_{s, p} , \wit{c_1})(\nu^{-1}\rtimes 1_{SL_2(\Q_p)})$
is thus irreducible subspace $\nu^{1}\rtimes 1_{SL_2(\Q_p)}=L(\nu^{1},\nu^{1};1).$

\noindent We discuss the case $\chi_p=1$ and $s=-2.$ In this case the
last two operators corresponding to $s$ and $c_2$ are holomorphic
isomorphisms, the pole occurs in the first operator, induced from the
$GL_2$-case $\nu^{-2}\times \nu^{-1}\to \nu^{-1}\times \nu^{-2}.$ Note
that the representation $\nu^{-2}\times \nu^{-1}\rtimes 1$ is of
length four (\cite{STSp4} Proposition 5.4 (i)) and we have (in the
appropriate Grothendieck group) $\nu^{-2}\rtimes
1_{SL_2(\Q_p)}=L(\nu^{3/2}St_{GL_2(\Q_p)};1)+L(\nu^{2},\nu^{1};1),$
from the Langlands classification it follows that
$L(\nu^{2},\nu^{1};1)$ is the unique subrepresentation of $\nu^{-2}\rtimes
1_{SL_2(\Q_p)}.$ On the other hand, the aforementioned pole is
happening on the quotient $\nu^{-3/2}St_{GL_2(\Q_p)}\rtimes 1.$
 We know that $\cal N(\Lambda_{s, p} , \wit{c_1})$ is holomorphic on
 $L(\nu^{2},\nu^{1};1)$ (this is the trivial character), and it has a pole
 on $L(\nu^{3/2}St_{GL_2(\Q_p)};1).$ 
\end{proof}

\begin{lem} 
\label{lem:c_1arch} Assume that $s\in R. $ Intertwining operator $\cal N(\Lambda_{s, \infty}, \wit{c_1})$ acting on the representation
 $\nu_{\infty}^{s} \rtimes 1_{SL_2(\R)}$ has poles precisely if $s<-1$ is an even integer and $\chi_{\infty}=1$ and if $s<-1$ is an odd integer and $\chi_{\infty}=sgn.$
\end{lem}
\begin{proof}
As in (\ref{decomp_c1}), we see that we might have a pole if $\chi_{\infty}$ is trivial or $sgn.$
\begin{itemize}
\item The first operator in the decomposition has a pole  if $s<-1$ is an  integer,  even if $\chi_{\infty}=1$ and odd if  $\chi_{\infty}=sgn.$
\item The second operator has a pole if $s<0$ is an integer, even if $\chi_{\infty}=sgn$ and odd if  $\chi_{\infty}=1.$
\item The third operator has a pole if $s<1$ is an integer, even if $\chi_{\infty}=1$ and odd if  $\chi_{\infty}=sgn.$
\end{itemize}

Note that if $s\le -1,$ the Langlands quotient $L(\chi_{\infty}\nu^{-s},\nu^1;1)$ is the unique irreducible subrepresentation of  $\chi_{\infty}\nu^s\rtimes 1_{SL_2({\R})}.$ If, in addition, $\chi_{\infty}=1$ we know that   $\cal N(\Lambda_{s, \infty} , \wit{c_1})$ is holomorphic and non-zero on  $L(\nu^{-s},\nu^1;1)$ because this is the spherical subquotient of the principal series. So we firstly resolve the case of $\chi_{\infty}=1.$

We first assume that $s=0$ (and $\chi_{\infty}=1$) (so that the third operator has a pole). Then, the first two operators are  holomorphic. We can now repeat the discussion from Lemma \ref{lem:c_1nonarch}, where we had a similar situation of (\ref{eq:s=0}). We again have the decomposition $\nu^0\rtimes 1_{SL_2(\R)}=L(\nu_{\infty}^{1/2}St_{GL_2(\R)};1)\oplus L(\nu_{\infty}^1;\nu_{\infty}^0\rtimes 1).$
Indeed, cf. Theorem 2.5 (ii) of (\cite{MuicSp4}) for the essentially square-integrable representation of $GL_2(\R)$ we denoted by $\nu^{1/2}St_{GL_2(\R)};,$ on the other hand, the representation $\nu_{\infty}^0\rtimes 1$ is irreducible (representation of $SL_2(\R)$ (cf. Theorem 2.4.(i) of \cite{MuicSp4}). For the  decomposition of $\nu_{\infty}^0\rtimes 1_{SL_2(\R)}$ 
(cf. Theorem 10.7., equation (10.72) of \cite{MuicSp4}). On the other hand, the length of the representation $\nu_{\infty}^0\times \nu_{\infty}^1\rtimes 1$ is six (\cite{MuicSp4}, Theorem 10.7 ), but the representation  $L(\nu_{\infty}^1;\nu_{\infty}^0\rtimes 1)$ is again a subrepresentation on this principal series, and appears with the multiplicity one, so we again can conclude that 
 $\cal N(\Lambda_{s, \infty}, \wit{c_1})$ is holomorphic on $ \nu_{\infty}^0 \rtimes 1_{SL_2(\R)}$ (and non-zero on each summand).

\noindent Now we examine the situation of $s=-1$ and $\chi_{\infty}=1.$ By \cite{MuicSp4}, equation (9.31) and Lemma 9.5. we see that the representation  $\nu_{\infty}^{-1} \rtimes 1_{SL_2(\R)}$ is irreducible. By our previous remark $\cal N(\Lambda_{s, \infty} , \wit{c_1})$ is holomorphic and non-zero on that representation, moreover, the image in   $\nu_{\infty}^{1} \times \nu_{\infty}^{-1}\rtimes 1$ generates an irreducible subrepresentation $\nu_{\infty}^{-1} \rtimes 1_{SL_2(\R)}(\cong \nu_{\infty}^{-1} \rtimes 1_{SL_2(\R)}).$

\noindent If $s<-1$ odd and $\chi_{\infty}=1,$  we see in Lemma 9.4 of \cite{MuicSp4} the representation  $\nu_{\infty}^{s} \rtimes 1_{SL_2(\R)}$ is irreducible, and the conclusion is clear.

\noindent If $s<-1$ is even and $\chi_{\infty}=1,$ we see that the representation  $\nu_{\infty}^{s} \rtimes 1_{SL_2(\R)}$ is reducible , of length two, (Theorem 11.1 (i) \cite{MuicSp4}). On the other subquotient (besides  $L(\chi_{\infty}\nu^{-s},\nu^1;1)$) of   $\nu_{\infty}^{s} \rtimes 1_{SL_2(\R)}$ the first intertwining operator has a pole, and the third operator does not vanish (on that subquotient). We conclude that  $\cal N(\Lambda_{s, \infty}, \wit{c_1})$ has a pole on that other subquotient (we could also argue as in the analogous non-archimedean case; namely it is easy to see that   $L(\chi_{\infty}\nu^{-s},\nu^1;1)$ cannot appear as a subrepresentation of $\nu^{-s}\times \nu^{-1}\rtimes 1$).

\noindent We examine the case when $s=-1$ and $\chi_{\infty}=sgn.$ According to Theorem 10.4 (ii) of \cite{MuicSp4}, the representation  $\nu_{\infty}^{-1} sgn\rtimes 1_{SL_2(\R)}$ is irreducible. The pole occurs for the third intertwining operator and it appears on the (induced) quotient of that intertwining operator, namely on the representation $\delta(1,2)\rtimes 1$ (here we use the notation of \cite{MuicSp4} (cf. Theorem 2.5 and Lemma 8.1). The representation $\delta(1,2)\rtimes 1$ decomposes as a sum of two tempered representations, so $\cal N(\Lambda_{s, \infty}, \wit{c_1})$ is holomorphic on  $\nu_{\infty}^{-1} sgn\rtimes 1_{SL_2(\R)}.$ The first and the second operator are holomorphic isomorphisms. Since $\cal N(\wit{c_1}(\Lambda_{s, \infty}), \wit{c_1})\cal N(\Lambda_{s, \infty}, \wit{c_1})=Id,$ and $N(\wit{c_1}(\Lambda_{s, \infty}), \wit{c_1})$ is holomorphic on $\nu_{\infty}^{1} sgn\times \nu^{-1}\rtimes 1$ it follows that $\cal N(\Lambda_{s, \infty}, \wit{c_1})$ is non-zero on $\nu_{\infty}^{-1} sgn\rtimes 1_{SL_2(\R)}.$

\noindent If $s<-1$ is odd integer and $\chi_{\infty}=sgn,$ the representation $\nu_{\infty}^ssgn\rtimes 1_{SL_2(\R)}$ is of the length two (Theorem 11.1.(i) of \cite{MuicSp4}). The first and the third operator have poles, and the second is holomorphic isomorphism. The first operator has a pole on the (induced) quotient $\delta(\nu^{-\frac{p+t}{2}},p-t)\rtimes 1$ and the third on $\delta(\nu^{\frac{p-t}{2}},p+t)\rtimes 1.$ Note that the latter  representation is a subquotient of the former (Theorem 10.3. of \cite{MuicSp4}).  Note that the Langlands quotient $L(\nu^{-s}sgn,\nu^1;1)$ is a subrepresentation of $\nu_{\infty}^ssgn\rtimes 1_{SL_2(\R)}$ and does not appear in the composition series of these two induced representations, so  $\cal N(\Lambda_{s, \infty}, \wit{c_1})$ is holomorphic on it. On the other hand, the ``other'' subquotient of $\nu_{\infty}^ssgn\rtimes 1_{SL_2(\R)}$ is $L(\delta(\nu^{\frac{p+t}{2}},p-t);1)$ and it appears with the multiplicity one in $\nu^{s}sgn \times \nu^1\rtimes 1$ as can be seen from Theorem 11.1. (i), (ii) and (iii) of \cite{MuicSp4}, so that there is a pole on the ``other'' subquotient of $\nu_{\infty}^ssgn\rtimes 1_{SL_2(\R)}$ (of order one).

\noindent If $s<-1$ is even integer and $\chi_{\infty}=sgn,$ the representation $\nu_{\infty}^ssgn\rtimes 1_{SL_2(\R)}$ is ireducible by Lema 9.4 of \cite{MuicSp4}. The first and the third  operator are holomorphic isomorphisms, and the second has a pole on the quotient $\nu^{-1}\rtimes  (X(-s,+)\oplus X(-s,-)),$ (e.g.,  cf. Lemma 7.2 of \cite{MuicSp4}). The same lemma guarantees that $\nu^{-1}\rtimes  X(-s,\varepsilon),\;\varepsilon=\pm,$ are irreducible (and are not isomorphic with  $\nu_{\infty}^ssgn\rtimes 1_{SL_2(\R)}$), so the intertwining operator  $\cal N(\Lambda_{s, \infty}, \wit{c_1})$ is holomorphic on $\nu_{\infty}^ssgn\rtimes 1_{SL_2(\R)},$ and similarly as above, we conclude that it is also non-zero there.

\end{proof}
Now we study the intertwining operators for the other relevant element of the Weyl group-which is similar, but easier that the previous case of $c_1.$

\begin{prop}
\label{prop:s} Let $p<\infty.$ The intertwining operator $\cal N(\Lambda_{s, p}, \wit{s}),$ where $s\in \R,$ is holomorphic on  $\nu_{p}^s\chi_p\rtimes 1_{SL_2(\Q_p)},$ unless $s=-2$ and $\chi_p=1;$  then it has a pole of the first order. If $p=\infty$ then we have poles precisely if $s<-1$ is even and $\chi_{\infty}=1$ and if $s<-1$ is odd and $\chi_{\infty}=sgn.$
\end{prop}
\begin{proof} Already proved in Lemma \ref{lem:c_1nonarch} and Lemma \ref{lem:c_1arch}
\end{proof}
\begin{prop} 
\label{prop:sc_1}
Let $p <\infty.$
 The intertwining operator $\cal N(\Lambda_{s, p}, \wit{sc_1})=\cal N(\Lambda_{s, p}, \wit{c_2s}),$ where $s ∈ \R,$  is
 holomorphic on $\chi_p \nu^s\rtimes 1_{SL_2(\Q_p)}$ unless $\chi_p =1$ and $s = −2.$ If $p = \infty$ then
 we have poles precisely if $s < −1$ is even and $\chi_{\infty} = 1$ and if $s < −1$ is odd
 and $\chi_{\infty} = sgn.$
\end{prop}
\begin{proof} Already proved in Lemma \ref{lem:c_1nonarch} and Lemma \ref{lem:c_1arch}
\end{proof}

Now we can explicitly describe the image of the constant term of the
degenerate Eisenstein series.

\begin{thm} \label{thm-heisenberg} Assume $s\ge 0.$ Then, the poles of the Eisenstein series can come only from the poles of the normalizing factors, and then only if $\chi=1$ and $s\in\{0,1,2\}$ (by well -known general results we know that there are no poles for $s=0$).
\begin{enumerate}
\item If $\chi=1$ and $s=0$ the global Eisenstein series is holomorphic (for any choice of $f=\otimes f_p \in \Ind_{GL_1(\A)\times SL_2(\A)}(\chi \nu^{s}\otimes 1)$). For $p\le \infty$ we have $\nu_p^0 \rtimes 1_{SL_2(\Q_p)}=L(\nu_p^1;\nu_p^0\rtimes 1)\oplus L(\nu^{1/2}St_{GL_2(\Q_p)};1).$  We choose a finite set of places $S'$ and we choose $f_p\in  L(\nu^{1/2}St_{GL_2(\Q_p)};1),$ for $p\in S',$ and for some bigger finite set of places $S\supset S'$ we choose $f_p\in L(\nu_p^1;\nu_p^0\rtimes 1)$ for $p\in S\setminus S'.$ For all the places outside of $S$ we choose $f_p$ to be the normalized spherical vector. Then, if $S'$ is even, this choice gives an automorphic realization of the corresponding (irreducible) global representation in the space of automorphic forms.

\item  If $\chi=1$ and $s=1$ the global Eisenstein series is holomorphic, and for every $p\le \infty,$ the representation
 $\nu_p^1 \rtimes 1_{SL_2(\Q_p)}$ is irreducible and (\ref{de-100}) gives the embedding of the irreducible  global representation $\Ind_{GL_1(\A)\times SL_2(\A)}(\chi \nu^{1}\otimes 1)$ in the space of automorphic forms.

\item  If $\chi=1$ and $s=2$ the global Eisenstein series has a pole of the first order, and, after removing the poles, (\ref{de-100}) gives the automorphic realization of the trivial representation in the space of (square-integrable) automorphic forms.
\item Assume  $s=0$ and $\chi^2=1$ but $\chi\neq 1$ then for all $p\le \infty,$ $\chi_p\rtimes 1_{SL_2(\Q_p)}=L(\nu_p^1;T_1)\oplus L(\nu_p^1;T_2)$, where $T_1$ and $T_2$ are non-isomorphic tempered representations (the limits of the disccrete series if $p=\infty$) such that $\chi_p\rtimes 1=T_1\oplus T_2$ in $SL_2(\Q_p).$ Here $T_1$ is such that $\cal N(s(\Lambda_{s, p}), \wit{c_2})$ acts on $\nu_p^{-1}\rtimes T_1$ as the identity, and on $\nu_p^{-1}\rtimes T_2$ as minus identity (so for the spherical places, the spherical vector belongs to  $L(\nu_p^1;T_1)$). Let $S$ be a finite set of places such that for $p\notin S,$ $f_p$ is spherical. Let $S'\subset S$ be such that for $p\in S'$ we choose $f_p\in  L(\nu_p^1;T_2)$ (and for $p\in S\setminus S'\; f_p\in  L(\nu_p^1;T_1)$). With these choices, the mapping (\ref{de-100}) gives an automorphic realization of this 
irreducible global representation if $|S'|$ is even.
\item In the rest of the cases (we still assume $s\ge 0$) which are not cover above, the embedding (\ref{de-100}) is holomorphic and gives an automorphic realization of the whole representation  $Ind_{GL_1(\A)\times SL_2(\A)}(\chi \nu^{s}\otimes 1).$
\end{enumerate}
\end{thm}
\begin{proof} Since s ≥ 0, in the expression (\ref{de-3}) all the local intertwining operators
are holomorphic (by Lemma \ref{lem:c_1nonarch}, Lemma \ref{lem:c_1arch}, Proposition \ref{prop:s} and Proposition \ref{prop:sc_1})  and the poles can only come from the normalizing factors,
and then, only if $\chi = 1$  (cf. Lemma \ref{lem:normalization}). We now assume that $\chi = 1.$

First assume $s=0.$  We note that then $s(\Lambda_s)=c_2s(\Lambda_s).$ We also claim that the intertwining operators $\cal N(\Lambda_{s, p}, \wit{c_2s})$ and $\cal N(\Lambda_{s, p}, \wit{s})$ have the same effect on $\nu^0\rtimes 1_{SL_2(\Q_p)}$ (for evry $p$). Indeed,  $\cal N(\Lambda_{s, p}, \wit{c_2s})=\cal N(s(\Lambda_{s, p}), \wit{c_2})\cal N(\Lambda_{s, p}, \wit{s}).$
Here the operator $\cal N(s(\Lambda_{s, p}), \wit{c_2})$ is induced from the $SL_2(\Q_p)$ operator $\nu_p^0\rtimes 1\to \nu_p^0\rtimes 1,$ which is holomorphic isomorphism (actually, an identity)  on an irreducible representation $\nu_p^0\rtimes 1\to \nu_p^0\rtimes 1$ (for all $p\le \infty$).

We also have that $c_1(\Lambda_s)=\Lambda_s$ for $s=0.$ Note that for every $p\le \infty,$ by   Lemma \ref{lem:c_1nonarch} and Lemma \ref{lem:c_1arch}, $\nu_p^0\rtimes  1_{SL_2(\Q_p)}=L(\nu_{p}^{1/2}St_{GL_2(\Q_p)};1)\oplus L(\nu_{p}^1;\nu_{p}^0\rtimes 1).$ and on $ L(\nu_{p}^1;\nu_{p}^0\rtimes 1)$ $\cal N(\Lambda_{s, p}, \wit{c_1})$ acts as the identity, and on $L(\nu_{p}^{1/2}St_{GL_2(\Q_p)};1)$ as minus identity.

We conclude that (\ref{de-3}) for $s=0$ becomes
\begin{gather}
\label{s=0}
E_{const}(f,s)=f_s+r(\Lambda_s,c_1)^{-1}(\otimes_{p\in S}\cal N(\Lambda_{s, p}, \wit{c_1})f_{p,s}\otimes (\otimes_{p\notin S}f_{c_1(s),p})+\\
(r(\Lambda_s,s)^{-1}+r(\Lambda_s,c_2s)^{-1})\otimes_{p\in S}\cal N(\Lambda_{s, p}, \wit{s})f_{p,s}\otimes (\otimes_{p\notin S}f_{s(s),p}).
\notag
\end{gather}
Here we denote by $s$ a real number (in this case equal to $0$) and an element of the Weyl group. Since we are dealing with the trivial global character over $\Q,$ all the global $\varepsilon$ factors are trivial, and we get that $r(\Lambda_s,s)^{-1}+r(\Lambda_s,c_2s)^{-1}=\frac{1}{L(s+2,1)}{(L(-s,1)+L(s,1))},$ so this expression is holomorphic and non-zero.
We conclude that $E_{const}(f,s)$ is holomorphic for $s=0$ for any choice of $f_s=\otimes f_{s,p}$ belonging to the global induced representation $Ind_{GL_1(\A)\times SL_2(\A)}(\chi \nu^{s}\otimes 1),$ but this is well-known from the general results of Langlands. But we can now describe the image of Eisenstein series in the space of automorphic forms. 

First we analyze the first row in (\ref{s=0}). Let $S'\subset S$ be a finite set of primes such that for them we choose $f_p\in L(\nu_{p}^{1/2}St_{GL_2(\Q_p)};1)$ and for rest of the places $S\setminus S'$ we chose $f_p\in L(\nu_{p}^1;\nu_{p}^0\rtimes 1).$ For $p\notin S$ we choose $f_p$ to be the normalized spherical vector. Then, we have  $\lim_{s\to 0}r(\Lambda_s,c_1)^{-1}=1,$ and $\cal N(\Lambda_{s, p}, \wit{c_1})f_{p,s}=-f_{p,s},\;p\in S',$  $\cal N(\Lambda_{s, p}, \wit{c_1})f_{p,s}=f_{p,s},\;p\notin S'.$
Thus, the first line of (\ref{s=0}) becomes $f_0+(-1)^{|S'|}f_0.$ We conclude that this is non-zero if $|S'|$ is even and this means that the global representation (a subrepresentation of  $Ind_{GL_1(\A)\times SL_2(\A)}(\chi \nu^{s}\otimes 1)$) with the local components consisting of $L(\nu_{p}^{1/2}St_{GL_2(\Q_p)};1)$ on an even number of places and $L(\nu_{p}^1;\nu_{p}^0\rtimes 1)$ on the rest of the places, is automorphic.

Now we analyze the second line. For $p\in S,$ the image $\cal N(\Lambda_{s, p}, \wit{s})f_{p,s}$ is non-zero if $f_{p,s}$ is not in the kernel of this operator (and the kernel is  $\nu^{-1/2}St_{GL_2(\Q_p)}\rtimes 1.$ Thus, if we pick $f_{p,s}$ from  $L(\nu_{p}^1;\nu_{p}^0\rtimes 1)$ for every $p\in S,$ we'll get a non-zero contribution.
This now means that the projection of the constant term to part of  the sum spanned by the images of the intertwining operators with respect to the Weyl group elements $s$ and $c_2s$ gives an automorphic realization of the global representation whose every local component is $L(\nu_{p}^1;\nu_{p}^0\rtimes 1)$-so we get only global representations which form a subset of the ones we obtained by analyzing the first line.

Now we analyze the case $s=1.$ Then, note that for each $p\le \infty$ the representation $\nu_p^1\rtimes 1_{SL_2(\Q_p)}$ is irreducible (as was noted in Lemma \ref{lem:c_1nonarch} and Lemma \ref{lem:c_1arch}), and thus spherical. In that case $c_2s(\Lambda_s)=sc_2s(\Lambda_s)=c_1(\Lambda_s).$ Moreover, $\cal N(\Lambda_{s, p}, \wit{sc_2s})=\cal N(c_2s(\Lambda_{s, p}), \wit{s})\cal N(\Lambda_{s, p}, \wit{c_2s}).$  Note that $\cal N(c_2s(\Lambda_{s, p}), \wit{s})$ is for $s=1$ identity operator induced from the $GL_2$ operator $\nu_p^{-1}\times \nu_p^{-1} \to \nu_p^{-1}\times \nu_p^{-1}.$
Then, we just sum $r(\Lambda_s,sc_2s)^{-1}+r(\Lambda_s,c_2s)^{-1}.$ Again we get that the poles cancel, and we obtain a non-zero holomorphic function for $s=1.$
Thus, the expression
\[(r(\Lambda_s,sc_2s)^{-1}+r(\Lambda_s,c_2s)^{-1}) \otimes_{p\in S}\cal N(\Lambda_{s, p}, \wit{c_2s})f_{p,s}\otimes (\otimes_{p\notin S}f_{c_2s(s),p})\]
gives a non-zero contribution to the constant term of Eisenstein series and this cannot cancel the identity contribution. This means that we have obtained an automorphic realization (through Eisenstein series) of the global (irreducible representation) $Ind_{GL_1(\A)\times SL_2(\A)}( \nu^{1}\otimes 1).$

For $s=2$ we have
\begin{gather*}
\lim_{s\to 2}(s-2)E_{const}(f,s)=(\lim_{s\to 2}(s-2)r(\Lambda_s,sc_2s)^{-1}) \otimes_{p\in S}\cal N(\Lambda_{s, p}, \wit{sc_2s})f_{p,s}\\
\otimes (\otimes_{p\notin S}f_{sc_2s(s),p}).
\end{gather*} We know that $\cal N(\Lambda_{s, p}, \wit{sc_2s})$ is holomorphic (and non-zero) on $\nu_p^2\rtimes 1_{SL_2(\Q_p)},$ for every $p\le \infty.$
Note that, because of the Langlands classification, $L(\nu_p^{2},\nu_p^{1};1)$ is the unique quotient of $\nu_p^2\rtimes 1_{SL_2(\Q_p)}.$ On the other hand $ \cal N(\Lambda_{s, p}, \wit{sc_2s})(\nu_p^2\rtimes 1_{SL_2(\Q_p)})\subset \nu_p^{-2}\times \nu_p^{-1}\rtimes 1$ and this representation has a unique irreducible subrepresentation; namely $L(\nu_p^{2},\nu_p^{1};1).$ This ensures that  $ \cal N(\Lambda_{s, p}, \wit{sc_2s})(\nu_p^2\rtimes 1_{SL_2(\Q_p)})=L(\nu_p^{2},\nu_p^{1};1),$ for every $p\le \infty.$ Note that $L(\nu_p^{2},\nu_p^{1};1)$ is actually a trivial representation, thus this (normalized) Eisenstein series gives a realization of the trivial representation as the irreducible subrepresentation in the space of square--integrable automorphic forms (since the constant term along the minimal parabolic has exponent $(-2,-1).$

If $\chi\neq 1,$ or $\chi=1$ but $s\notin\{0,1,2\}$ we have the following situation.
Let $w\in W,\; w(2e_2)>0.$ Then $w(\Lambda_s)\neq \Lambda_s,$ unless $\chi^2=1$ and $s=0;$ then $sc_2s(\Lambda_s)=\Lambda_s.$
So if $\chi^2\neq 1$ or $s\neq 0,$ (and we are not in the situations already covered) nothing will cancel the identity contribution; this gives the automorphic realization of the whole representation  $Ind_{GL_1(\A)\times SL_2(\A)}(\chi \nu^{s}\otimes 1).$

Now assume $s=0$ and $\chi^2=1$ but $\chi\neq 1.$ Then, for every $p\le \infty$ such that $\chi_p\neq 1,$ the representation $\chi_p\rtimes 1$  (of $SL_2(\Q_p)$) is reducible, and sum of two irreducible tempered representations, say $T_1$ and $T_2$ (the limits of the discrete series if $p=\infty$).  Note that $\Lambda_s=sc_2s(\Lambda_s)$ and $s(\Lambda_s)=c_2s(\Lambda_s).$ Assume that on $T_1$ the normalized intertwining operator $ \chi_p\rtimes 1\to \chi_p\rtimes 1$ acts as the identity and on $T_2$ as -identity. We also have that $\chi_p\rtimes 1_{SL_2(\Q_p)}=L( \nu_p^1;T_1)\oplus L( \nu_p^1;T_2)$ for all $p$ (\cite{STSp4}, Proposition 5.4 and \cite{MuicSp4} Lema 9.6). We conclude that $\cal N(\Lambda_{s, p}, \wit{s})$ acts on $\chi_p\rtimes 1_{SL_2(\Q_p)}$ as an holomorphic isomorphism, and then $\cal N(s(\Lambda_{s, p}), \wit{c_2})$ acts on $\nu^{-1}\rtimes T_1$ as identity, and on $\nu^{-1}\rtimes T_2$ as minus identity. Now again $\cal N(c_2s(\Lambda_{s, p}), \wit{s})$ is a holomorphic isomorphism. Let $S$ be a finite set of places such that for $p\notin S,$ $f_p$ is the normalized spherical vector. For a subset $S'\subset S$ such that for every $p\in S'$ $\chi_p\neq 1,$ we choose $f_p$ to belong to  $L( \nu_p^1;T_2)$ (and $f_p$ belongs to  $L( \nu_p^1;T_1)$ for $p\in S\setminus S'$).
We then have

\begin{gather}
\label{eq:Heisen1}
E_{const}(f,s)=f_s+r(\Lambda_s,sc_2s)^{-1}(-1)^{|S'|}f_s+\\
(r(\Lambda_s,s)^{-1}+(-1)^{|S'|}r(\Lambda_s,c_2s)^{-1})\otimes_{p\in S}\cal N(\Lambda_{s, p}, \wit{s})f_{p,s}\otimes (\otimes_{p\notin S}f_{s(s),p}).
\notag
\end{gather}
We use the functional equation (\cite{Ramakrishnan_Valenza}, p.279,)  $L(1-s,\chi^{-1})=\varepsilon(s,\chi)L(s,\chi).$ Now, according to that the expression for $r(\Lambda_s,sc_2s)^{-1}$ becomes
\[r(\Lambda_s,sc_2s)^{-1}=\frac{L(-s+2,\chi)\varepsilon(-s+2,\chi)}{L(s+2,\chi)\varepsilon(s+2,\chi)\varepsilon(s,\chi)\varepsilon(s+1,\chi)}.\]
Since $L(s,\chi)$ is holomorphic for $s\in \C$ we have $\lim_{s\to 0}\frac{L(-s+2,\chi)\varepsilon(-s+2,\chi)}{L(s+2,\chi)\varepsilon(s+2,\chi)}=1$ so that $\lim_{s\to 0}r(\Lambda_s,sc_2s)^{-1}=\frac{1}{\varepsilon(0,\chi)\varepsilon(1,\chi)}.$ Further, we have $L(-s,\chi)=\varepsilon(1+s,\chi)L(1+s,\chi).$ If we multiply these two functional equations and let $s=0$, we get $\varepsilon (s,\chi)\varepsilon(s+1,\chi)=1$ (since $L(0,\chi)L(1,\chi)\neq 0$).
Then, the factor in the first line of (\ref{eq:Heisen1}) becomes $1+(-1)^{|S'|}.$ We conclude that if $|S'|$ is odd that the first line of  (\ref{eq:Heisen1}) vanishes. The  numerical factor in the second line of (\ref{eq:Heisen1}) becomes
$\frac{1}{L(s+2,\chi)\varepsilon(s+2,\chi)\varepsilon(s+1,\chi)}(L(-s,\chi)+(-1)^{|S'|}L(s,\chi)).$ We conclude that for $|S'|$ odd the second line also vanishes. So,we can get a  non-zero contribution only if $|S'|$ is even.
 \end{proof}

\begin{thm} \label{thm-heisenberg-1} Assume $s<0.$ 
\begin{enumerate}
\item If $-1<s<0$  the Eisenstein series is holomorphic and  (\ref{de-100}) gives an embedding of the irreducible 
representation  $Ind_{GL_1(\A)\times SL_2(\A)}(\chi \nu^{s}\otimes 1)$ in the space of automorphic forms. 
\item If $s=-1$ and $\chi\neq 1$ the result is analogous to the previous case.
\item If $s=-1$ and $\chi=1$  the Eisenstein series is identically zero on $Ind_{GL_1(\A)\times SL_2(\A)}( \nu^{-1}\otimes 1).$
\item If $-2<s<-1$ all the (inverses) of the (non-trivial) normalizing factors can have a pole  (of the same order) comming from the zero of the $L$--function in the denominator. After the normalization, we have an embedding of the representation  $Ind_{GL_1(\A)\times SL_2(\A)}(\chi \nu^{s}\otimes 1)$ in the space of automorphic forms.
\item  If $s=-2$ and $\chi=1$ then, if we pick  $f_p\in L(\nu^2,\nu^1;1)$ (the trivial representation) for all $p,$ then
$E_{const}(f,-2)=f_{-2}$ so  (\ref{de-100}) gives an embedding of the trivial representation in the space of (square-integrable) automorphic forms. If we pick at exactly one finite place $f_p\in L(\nu_p^{3/2}St_{GL_2(\Q_p)};1)$ (and on the rest of the places the trivial representation), the Eisenstein series are holomorphic on this global representation and the image is, on that one place, an irreducible subrepresentation isomorphic to $L(\nu_p^{3/2}St_{GL_2(\Q_p)};1),$ on the rest of the places the image spans a representation of the length two (in semisimplification $L(\nu^2,\nu^1;1)+L(\nu_p^{3/2}St_{GL_2(\Q_p)};1)$).
Analogously, if we pick at the finite number of place, say $|S|,$ a vector from $L(\nu_p^{3/2}St_{GL_2(\Q_p)};1),$ the Eisenstein series have a pole of order $|S|-1$ (so we can get a pole of any finite order), and after normalization, the local images are irreducible for the ramified choice of subquotient, and of the length equal to two for the choice of the unramified subquotient.
\item   If $s=-2$ and $\chi\neq 1$ we have the following. For each $p$ such that $\chi_p=1$ the local intertwining operators have a pole if $f_p\in  L(\nu_p^{3/2}St_{GL_2(\Q_p)};1),$ and if $\chi_p\neq 1$ there are no poles for the intertwining operators (all isomorphisms)  and the image is isomorphic to the representation $\chi_p\nu_p^{-2}\rtimes 1_{SL_2(\Q_p)}.$  The discussion about the image is analogous to the previous case when $\chi_p=1,$ and we can obtain a pole of order $|S|$ if, for every $p\in S$ $\chi_p=1$ and $f_p\in L(\nu_p^{3/2}St_{GL_2(\Q_p)};1).$

\item Assume $s<-2.$ Then unless $s$ is even integer and $\chi_{\infty}=1$ or $s$ is odd integer and $\chi_{\infty}=sgn,$ the Eisenstein series is always holomorphic and (\ref{de-100}) gives an embedding of  $Ind_{GL_1(\A)\times SL_2(\A)}(\chi \nu^{s}\otimes 1)$ in the space of automorphic forms. If  $s$ is even integer and $\chi_{\infty}=1$ or $s$ is odd integer and $\chi_{\infty}=sgn,$ and (in both of these cases) we pick $f_{\infty}\in L(\delta\nu^{\frac{-s+1}{2}}, -s-1)$ (notation \cite{MuicSp4} Theorem 11.1(i)), the Eisenstein series have a pole on $\otimes f_p$ of the first order. After the normalization, the image on the non-archimedean place spans an irreducible representation, and on the archimedean place it spans a representation $\chi_{\infty}\nu^{-s}\rtimes 1_{SL_2(\R)}.$  
\end{enumerate}
\end{thm}
\begin{proof}
Assume $-1<s<0.$ The claim is then obvious.

In the case $s=-1$  if $\chi\neq 1$ the Eisenstein series is holomorphic, the other contributions in the constant term cannot cancel the identity contribution, the global representation is  again irreducible (Lemma 9.5. or Theorem 10.4 of 
\cite{MuicSp4}), so again (\ref{de-100})  gives an embedding in the space of automorphic forms.

On the other hand, if $s=-1$ and $\chi=1,$ the global representation is again irreducible, and intertwining operators holomorphic) but the normalizing factors $r(\Lambda_s, c_1)^{-1}$ and $r(\Lambda_s,c_2s)^{-1}$ vanish. On the other hand, $\lim_{s\to -1}r(\Lambda_s,s)^{-1}=-\frac{1}{\varepsilon (1,1)}=-1$ and $\cal N(\Lambda_{s, p}, \wit{s})f_{p,s}$ is the identity. In this way, (\ref{de-3}) comes down to
\[E_{const}(f,-1)=0.\]
Now assume $-2<s<-1.$ Then, the normalized intertwining operators  are all holomorphic and the representation  $Ind_{GL_1(\A)\times SL_2(\A)}(\chi \nu^{s}\otimes 1)$ is irreducible (obvious, e.g., Theorem 12.1. (ii) \cite{MuicSp4}). The (inverses of )normalizing factors might have a pole (and the order of the possible pole is equal to the order of the zero of $L(s+2,\chi),$ and this happens for all the (non-trivial) normalizing factors. Note that, for example $sc_2s(\Lambda_s)\neq s(\Lambda_s),$ so after the possible removing of the pole, the contributions according to the different element of the Weyl group do not cancel, so by (\ref{de-3}) we have an embedding of this global representation into the space of automorphic forms.

Assume $s=-2 $ and $\chi=1.$ Then  all the nontrivial normalizing factors vanish because of the poles in the denominator. On the other hand, if we pick $f_p\in L(\nu^2,\nu^1;1)$ for all $w$ appearing in the expression for the constant term,  $\cal N(\Lambda_{s, p}, \wit{w})f_p$ is holomorphic and non-zero, and if we pick $f_p\in L(\nu^{3/2}St_{GL_2(\Q_p)};1)$ we have a pole for each such $w,$ analogously for the archimedean places (Lemma \ref{lem:c_1nonarch}, Lemma \ref{lem:c_1arch}, Proposition \ref{prop:s}, Proposition \ref{prop:sc_1}). We conclude that by picking $f_p \in L(\nu^2,\nu^1;1,)$ for each $p\le \infty,$
the contributions in the constant term of Eisenstein series (from all the non-trivial elements of the Weyl group appearing there)
will be zero,so that we would have $E_{const}(f,-2)=f_{-2}.$ This would imply that the global representation consisting of  $L(\nu_p^2,\nu_p^1;1,)$ for every $p\le \infty$ appears in the space of automorphic forms through (\ref{de-100}), moreover the appearance in $\cal A\left(Sp_4(\Q)\backslash Sp_4(\A)\right)$ is in the space of square--integrable automorphic forms
(this is well-known); $L(\nu_p^2,\nu_p^1;1)$ is the trivial character). On the other hand, assume  we pick  $f_p\in L(\nu^{3/2}St_{GL_2(\Q_p)};1)$ at exactly one (say, finite) place from $S.$ Then, the contribution $r(\Lambda_s, w)^{-1}\otimes \otimes_{p\in S}\cal N(\Lambda_{s, p}, \wit{w})f_{p,s}\otimes (\otimes_{p\notin S}f_{w(s),p})$ is holomorphic (for each $w\neq 1).$  To study the image, we examine what is happening on each place. On that specific  place, $f_p$ spans the representation $L(\nu^{3/2}St_{GL_2(\Q_p)};1)$ which is s subrepresentation in  $Ind_{B}^{Sp_4(\Q_p)}(w( \nu_p^{-2}\otimes \nu_p^{-1})).$ On the other places (in $S$ and outside of $S$)) it is enough to see what kind of space the spherical vector generates. It generates a subspace of length  two in  $Ind_{B}^{Sp_4(\Q_p)}(w( \nu_p^{-2}\otimes \nu_p^{-1}))$ where $w\in \{s,c_2s,sc_2s\}$ : it cannot be an irreducible subrepresentation in any of those induced representations (a simple Jacquet module argument); on the other hand it generates a subrepresentation of length two in  $Ind_{B}^{Sp_4(\Q_p)}(sc_2s( \nu_p^{-2}\otimes \nu_p^{-1})),$ and the other representations are isomorphic to it. The generated representation is (in semisimplification) $L(\nu_p^2,\nu_p^{1};1)+  L(\nu^{3/2}St_{GL_2(\Q_p)};1).$  Analogously, by choosing a arbitrary but finite set of places $p$  where we can pick   $f_p\in L(\nu^{3/2}St_{GL_2(\Q_p)};1)$ we can make  the pole for $s=-2$ of $E_{const}(s,f)$ of the arbitrary high order. After removing the poles  the local images are irreducible on the ramified choices and of length two on the unramified choices.  Thus, the image of (\ref{de-100})  spans (higly) reducible representation. Similarly, if $s=-2$ but $\chi\neq 1,$ all the normalizing factors are holomorphic and non-vanishing, and we have a pole for each place $p$ for which $\chi_p=1,$ so the  discussion about local images  is similar to the discussion for $\chi=1$--case, modulo the  number of places where we have $\chi_p=1.$ If $\chi_p\neq 1,$ the representation $\chi_p\nu_p^{-2}\rtimes 1_{SL_2(\Q_p)}$ is irreducible, (e.g., Lemma 9.4. of \cite{MuicSp4}, equally easy for the non-archimedean case) and all the intertwining operators are holomorphic isomorphisms.

If $s<-2$ all the normalizing factors are holomorphic and non-zero, and all the representatins on the finite places are irreducible, and the intertwining operators acting on those representations are holomorphic isomorphisms. Then, the images generated locally are also irreducible. We can only get poles of the intertwining operators at the archimedean place, when $s$ is odd and $\chi_{\infty}=sgn$ or when  $s$ is even and $\chi_{\infty}=1$ for all of them (attached to $s,\;c_2s,\;sc_2s$) as follows from Lemma \ref{lem:c_1arch}. In all other situation the Eisenstein series are holomorphic and (\ref{de-100}) gives the embedding in the space of automorphic forms.
In these exceptional cases, we get a pole of the first order if we pick $f_{\infty}\in L(\delta \nu^{\frac{-s+1}{2}},-s-1;1)$ in both of these cases (Theorem 11.1.(i) of \cite{MuicSp4}). By normalizing the Eisenstein series to eliminate this pole, we again get no cancelations between contributions attached to the different Weyl group elements. In the situation when we pick $f_{\infty}\in L(\delta \nu^{\frac{-s+1}{2}},-s-1;1)$ (and normalize the Eisenstien series), the image will generate an irreducible representation in $Ind_B^{Sp_4(\R)}w(\nu^s\otimes \nu{-1}),$ for $w\in \{s,c_2s\},$ but reducible for $w=sc_2s;$ indeed, for  $w=sc_2s$ $f_{\infty}$ generates the whole representation $\chi_{\infty}\nu_{\infty}^{-s}\rtimes 1_{SL_2(\R)}\hookrightarrow \nu_{\infty}^{-s}\times \nu^{-1}\rtimes 1.$ 
\end{proof}

\section{The Siegel case}\label{siegel}
We now study the degenerate Eisenstein series of the representation 
\[Ind_{P_2(\A)}^{Sp_4(\A)}(\chi\nu^s1_{GL_2(\A)}).\]
Thus $\Lambda_s=\chi\nu^{s-1/2}\otimes \chi\nu^{s+1/2},$ where $s\in \R.$

The terms in (\ref{de-3}) which appear in the Siegel case are the following:
\begin{equation}
\label{eq:W'}
W'=\{w\in W:w(e_1-e_2)>0\}=\{id,c_2,sc_2,c_2sc_2\}.
\end{equation}
The corresponding normalizing factors are 
\[r(\Lambda_s,c_2)^{-1}=\frac{L(s+\frac{1}{2},\chi)}{L(s+\frac{3}{2},\chi)\varepsilon(s+\frac{3}{2},\chi)},\]
\[r(\Lambda_s,sc_2)^{-1}=\frac{L(s+\frac{1}{2},\chi)L(2s,\chi^2)}{L(s+\frac{3}{2},\chi)\varepsilon(s+\frac{3}{2},\chi)L(2s+1,\chi^2)\varepsilon (2s+1,\chi^2)},\]
\[r(\Lambda_s,c_2sc_2)^{-1}=\frac{L(2s,\chi^2)L(s-\frac{1}{2},\chi)}{L(s+\frac{3}{2},\chi)L(2s+1,\chi^2)\varepsilon(s+\frac{3}{2},\chi)\varepsilon(s+\frac{1}{2},\chi)\varepsilon(2s+1,\chi^2)}.\]

We discuss the poles of the (inverses) of the normalizing factors above, which appear only in the following situations.
\begin{prop}
\label{prop:Siegel_normalizations}
\begin{enumerate}
\item Assume $s\ge 0.$ 
\begin{enumerate}
\item  $r(\Lambda_s,c_2)^{-1}$   has a pole of the first order for $s=\frac{1}{2}$ if $\chi=1.$ 
\item 
$r(\Lambda_s,sc_2)^{-1}$ has a pole  for $s=\frac{1}{2}$ of the second order if $\chi=1$ and of the first order if $\chi^2=1,$ but $\chi\neq 1.$
\item  for $r(\Lambda_s,c_2sc_2)^{-1}$ the same discussion as for  $r(\Lambda_s,sc_2)^{-1}.$
\end{enumerate}
\item Assume $s<0.$
\begin{enumerate}
\item $r(\Lambda_s,c_2)^{-1}$ may have a pole  for some $-\frac{3}{2}<s<-\frac{1}{2},$ where the denominator has a zero.
\item $r(\Lambda_s,sc_2)^{-1}$ may have a pole  for some $-\frac{3}{2}<s<-\frac{1}{2},$ or some $-\frac{1}{2}<s<0$ where the
denominator has a zero.
\item for $r(\Lambda_s,c_2sc_2)^{-1}$ the same discussion as for  $r(\Lambda_s,sc_2)^{-1}.$
\end{enumerate}
\end{enumerate}
\end{prop}
Now we discuss the local normalized intertwining operators corresponding to the elements of $W'$ (\ref{eq:W'}). The following lemma follows straightforward from the $GL_2$ and $SL_2$ cases.
\begin{lem}
\label{lem:operatorsSiegel}
\begin{enumerate}
\item Assume $p<\infty.$ Then, the operator $\cal N(\Lambda_{s, p}, \wit{c_2})$ is holomorpic unless $s=-\frac{3}{2}$ and $\chi_p=1.$ Assume $p=\infty.$ Then, the operator $\cal N(\Lambda_{s, \infty}, \wit{c_2})$ is holomorpic unless $s+\frac{1}{2}$ is a negative odd integer and $\chi_{\infty}=1$ and  $s+\frac{1}{2}$ is a negative even  integer and $\chi_{\infty}=sgn.$

\item In addition to the poles from the case of  $\cal N(\Lambda_{s, p}, \wit{c_2})$, the operator  $\cal N(\Lambda_{s, p}, \wit{sc_2})$ may have additional poles as follows: if $p<\infty$ then the pole might appear if $\chi_p^2=1$ and $s=-\frac{1}{2}.$ If $p=\infty$, then the pole appears if $2s$ is a negative odd integer, and $\chi_{\infty}=sgn^l,\;l\in \{0,1\}.$
\item In addition to the poles from the case of  $\cal N(\Lambda_{s, p}, \wit{sc_2})$, the operator  $\cal N(\Lambda_{s, p}, \wit{c_2sc_2})$ may have additional poles as follows: if $p<\infty$ and $s=-\frac{1}{2}$ and $\chi_p=1.$ Assume $p=\infty.$
Then we may have a pole if   $s-\frac{1}{2}$  a negative odd integer and  $\chi_{\infty}=1$ and  $s-\frac{1}{2}$ is a negative even  integer and $\chi_{\infty}=sgn.$
\end{enumerate}
\end{lem}
\begin{thm}
\label{thm:image_Siegel}
Assume $s\ge 0.$ Then all the local intertwining operators correspoding to $w\in W'$ are holomorphic.
\begin{enumerate}
\item  If $s\neq \frac{1}{2},$ or $\chi^2\neq 1,$ then the Eisenstein series are holomorphic for every choice of $f=\otimes f_p,$ and (\ref{de-100}) gives and holomorphic embedding of   $Ind_{P_2(\A)}^{Sp_4(\A)}(\chi\nu^s1_{GL_2(\A)})$ in the space of automorphic forms.
\item  If $s=\frac{1}{2}$ and $\chi=1,$ $E_ {const}(s,f)$ has a pole of the first order, and then the image of (\ref{de-100}) is the unique spherical subquotient of $Ind_{P_2(\A)}^{Sp_4(\A)}(\chi\nu^{\frac{1}{2}}1_{GL_2(\A)}).$
\item  If $s=\frac{1}{2}$ and $\chi\neq 1,$ then for all $p\le \infty$ such that $\chi_p\neq 1$, let  $\chi_p\nu_p^0\rtimes 1=T_1\oplus T_2$ (of course, then $T_1$ and $T_2$ depend on $p$). Then, in the appropriate Grothendieck group we have
\[\nu_p^{1/2}\chi_p1_{GL_2(\Q_p)}\rtimes 1=L(\chi_p\nu^1;T_1)+L(\chi_p\nu^1;T_2)+L(\chi_p\nu_p^{1/2}St_{GL_2(\Q_p)};1).\]
Here $T_2$ is the representation on which $SL_2(\Q_p)$ operator on   $\chi_p\nu_p^0\rtimes 1$ attains value minus identity. Now, we pick $f_p$ from the subquotient $L(\chi_p\nu^1;T_2)$ on the set of places $S'.$ If $|S'|$ is even, then $E_{const}(\frac{1}{2},\cdot)$ has a pole of the first order, and (\ref{de-100}) gives, in the space of automorphic forms, a realization of the global representation with $L(\chi_p\nu^1;T_2)$ on the even number of places, and $L(\chi_p\nu^1;T_1)$ on the rest of the places.
 If $|S'|$ is odd, then  $E_{const}(\frac{1}{2},\cdot)$ is holomorphic and (\ref{de-100}) gives an embedding of the whole global representation 
$Ind_{P_2(\A)}^{Sp_4(\A)}(\chi\nu^{\frac{1}{2}}_{GL_2(\A)})$ in the space of automorphic forms.
\end{enumerate}
\end{thm}
\begin{proof}
The claims for $s\neq \frac{1}{2}$ or $\chi^2\neq 1$ are obvious. Now assume that $\chi^2=1$  and $s=\frac{1}{2}.$ 
Since for every $p\le \infty$ $\chi_p^2=1,$ we have that  $\cal N(\Lambda_{s, p}, \wit{sc_2})$ and $\cal N(\Lambda_{s, p}, \wit{c_2sc_2})$ have the same codomain. Assume now $p<\infty.$ If $\chi_p=1$ then $\nu_p^{1/2}1_{GL_2(\Q_p)}\rtimes 1$ has the unique irreducible quotient, namely $L(\nu_p^1;\nu_p^0\rtimes1)$ (which is spherical) (cf. \cite{STSp4} Proposition 5.4.(ii)) and a tempered subrepresentation, say $T_2.$ Then, the spherical vector generates the whole representation  $\nu_p^{1/2}1_{GL_2(\Q_p)}\rtimes 1,$ and since   $\cal N(\Lambda_{s, p}, \wit{sc_2})$ and  $\cal N(\Lambda_{s, p}, \wit{c_2sc_2})$ agree on the spherical vector (and are holomorphic and non-zero on that vector), they agree on the whole representation $\nu_p^{1/2}1_{GL_2(\Q_p)}\rtimes 1.$ The image of this spherical vector by $\cal N(\Lambda_{s, p}, \wit{sc_2})$ (or by $\cal N(\Lambda_{s, p}, \wit{c_2sc_2}))$ generates in $\nu_p^{-1}\times \nu_p^0\rtimes$ an irreducible subspace (thus isomorphic to $L(\nu_p^1;\nu_p^0\rtimes1).$  If $\chi_p\neq 1,$ then $\nu_p^{1/2}\chi_p1_{GL_2(\Q_p)}\rtimes 1$ is of length 3, and in the appropriate Grothendieck group we have 
\[\nu_p^{1/2}\chi_p1_{GL_2(\Q_p)}\rtimes 1=L(\chi_p\nu^1;T_1)+L(\chi_p\nu^1;T_2)+L(\chi_p\nu_p^{1/2}St_{GL_2(\Q_p)};1).\]
Here $\chi_p\nu_p^0\rtimes 1=T_1\oplus T_2$ and  $L(\chi_p\nu_p^{1/2}St_{GL_2(\Q_p)};1)$ is an irreducible subrepresentation of $\nu_p^{1/2}\chi_p1_{GL_2(\Q_p)}\rtimes 1,$ moreover,  we have an epimorphism $\nu_p^{1/2}\chi_p1_{GL_2(\Q_p)}\rtimes 1 \to L(\chi_p\nu^1;T_1)\oplus L(\chi_p\nu^1;T_2)$ (with the kernel  $L(\chi_p\nu_p^{1/2}St_{GL_2(\Q_p)};1)).$ Since  $\cal N(\Lambda_{s, p}, \wit{sc_2})$ maps $\nu_p^{1/2}\chi_p1_{GL_2(\Q_p)}\rtimes 1$ to the represenation $\chi_p\nu_p^{-1}\times \nu_p^0\chi_p\rtimes 1,$ which has two unique subrepresentations $L(\chi_p\nu^1;T_1)$ and  $L(\chi_p\nu^1;T_2),$ we see that $L(\chi_p\nu_p^{1/2}St_{GL_2(\Q_p)};1)$ is in the kernel of the operator  $\cal N(\Lambda_{s, p}, \wit{sc_2})$ (restricted on $\nu_p^{1/2}\chi_p1_{GL_2(\Q_p)}\rtimes 1$) and it is easy to see that   
\[\cal N(\Lambda_{s, p}, \wit{sc_2})(\nu_p^{1/2}\chi_p1_{GL_2(\Q_p)}\rtimes 1)=L(\chi_p\nu^1;T_1)\oplus L(\chi_p\nu^1;T_2).\]  On the other hand
\[\cal N(\Lambda_{s, p}, \wit{c_2sc_2})= \cal N(sc_2(\Lambda_{s, p}), \wit{c_2})\cal N(\Lambda_{s, p}, \wit{sc_2}),\]
 and 
$\chi_p\nu_p^{-1}\times \chi_p\nu_p^{0} \rtimes 1=\chi_p\nu_p^{-1}\rtimes T_1 \oplus \chi_p\nu_p^{-1}\rtimes T_2,$ and $\cal N(sc_2(\Lambda_{s, p}), \wit{c_2})$ acts on one of these two summands as identity and on the other as minus identity, say on the first one as identity. 
Assume now $p=\infty.$ If $\chi_{\infty}=1$ then again as in the non-archimedean case, the operators  $\cal N(\Lambda_{s, p}, \wit{sc_2})$ and $\cal N(\Lambda_{s, p}, \wit{c_2sc_2})$ have the same action on $\nu_{\infty}^{1/2}\chi_{\infty}1_{GL_2(\R)}\rtimes 1$ (which is again generated by the spherical vector), and the image is isomorphic to  $L(\nu_{\infty}^1;\nu_{\infty}^0\rtimes1).$ If $\chi_{\infty}=sgn$ we have the same situation as in the non-archimdedan case; cf. Theorem 11.2 of \cite{MuicSp4}. 

Now assume that $\chi=1.$ Then $E_{const}(1/2,\cdot)$ has a pole of the first order (due to the poles of the global normalizing factors) and (normalized) (\ref{de-3}) becomes
\begin{gather}
\label{eq1}
\lim_{s\to \frac{1}{2}}(s-\frac{1}{2})(r(\Lambda_s,sc_2)^{-1}+r(\Lambda_s,c_2sc_2)^{-1})\otimes_{p\in S}\cal N(\Lambda_{s, p}, \wit{sc_2})f_{p,s} \otimes \otimes_{p\notin S} f_{p,sc_2(s)}+\\
 \lim_{s\to \frac{1}{2}}(s-\frac{1}{2})(r(\Lambda_s,c_2)^{-1}\otimes_{p\in S}\cal N(\Lambda_{s, p}, \wit{sc_2})f_{p,s} \otimes \otimes_{p\notin S} f_{p,sc_2(s)}.
\notag
\end{gather}
Note that for $\cal N(\Lambda_{s, p}, \wit{c_2}),$ the image is again generated by the spherical vector, thus for $\cal N(\Lambda_{s, p}, \wit{c_2}),$ all other subquotients of $\nu_{p}^{1/2}\chi_{p}1_{GL_2(\Q_p)}\rtimes 1,\,p\le \infty$ are in the kernel.
Because $\chi=1,$ all the $\varepsilon$ factors are trivial, and $\lim_{s\to \frac{1}{2}}(s-\frac{1}{2})(r(\Lambda_s,sc_2)^{-1}+r(\Lambda_s,c_2sc_2)^{-1})$ becomes $\lim_{s\to \frac{1}{2}}(s-\frac{1}{2})\frac{L(2s,1)}{L(2,1)^2}\lim_{t\to 0}(L(-t,1)+L(t,1)).$
In the equation (\ref{eq1}) the contribution is zero if we pick a vector $f_p$ which does not belong to the spherical quotient.

Assume $\chi\neq 1,$ but $\chi^2=1.$ Assume that we have picked $S'\subset S$ such that for every $p\in S'$ $\chi_p\neq 1$ and  $f_p$  belongs to the subquotient $L(\chi_p\nu^1;T_2).$
We denote $A(s)=\frac{L(2s,\chi^2)}{L(s+\frac{3}{2},\chi)\varepsilon(s+\frac{3}{2},\chi)L(2s+1,\chi^2)\varepsilon (2s+1,\chi^2)}.$ Then $r(\Lambda_s,sc_2)^{-1}=A(s)L(s+\frac{1}{2},\chi)$ and $r(\Lambda_s,c_2sc_2)^{-1}=A(s)\frac{L(s-\frac{1}{2},\chi)}{\varepsilon (s+\frac{1}{2},\chi)}.$ Then, it is easy to see that the contributions to $E_{const}$ comming from $sc_2$ and $c_2sc_2$ together give:

\[\frac{A(s)(L(\frac{1}{2}-s,\chi)+L(s-\frac{1}{2},\chi)(-1)^{|S'|})}{\varepsilon(s+\frac{1}{2},\chi)}\otimes _{p\in S}\cal N(\Lambda_{s, p}, \wit{sc_2})f_{p,s} \otimes \otimes_{p\notin S} f_{p,sc_2(s)}.\]
If $|S'|$ is odd, since $A(s)$ has a pole for $s=\frac{1}{2},$ the expression 
\[A(s)(L(\frac{1}{2}-s,\chi)+L(s-\frac{1}{2},\chi)(-1)^{|S'|})\]
 is then holomorphic and non-zero, and if $|S'|$ is even it has a pole. Note that for $\chi_p\neq 1,$ the operator $\cal N(\Lambda_{s, p}, \wit{c_2})$ is a holomorphic isomorphism on $\nu_p^{1/2}\chi_p1_{GL_2(\Q_p)}\rtimes 1.$

We conclude: if $|S'|$ is even, then $E_{const}(\frac{1}{2},\cdot)$ has a pole of the first order, and (\ref{de-100}) gives, in the space of automorphic forms, a realization of the global representation with $L(\chi_p\nu^1;T_2)$ on the even number of places, and $L(\chi_p\nu^1;T_1)$ on the rest of the places.
Note that if $|S'|$ is odd, then  $E_{const}(\frac{1}{2},\cdot)$ is holomorphic and (\ref{de-100}) gives an embedding of the whole global representation 
$Ind_{P_2(\A)}^{Sp_4(\A)}(\chi\nu^{\frac{1}{2}}_{GL_2(\A)})$ in the space of automorphic forms.
\end{proof}

\begin{thm}
\label{thm:image_Siegel_s<0}
Assume now that $s<0.$ Then we have the following:
\begin{enumerate}
\item Assume $-\frac{1}{2}<s<0.$ Then either the Eisenstein series are holomorphic or they have a possible pole of the first order due to the pole of $r(\Lambda_s,sc_2)^{-1}$ and  $r(\Lambda_s,c_2sc_2)^{-1}$ coming from the zero of  $L(2s+1,\chi^2)$ for given $s;$  in both cases the intertwining operators are holomorphic isomorphisms so that  (\ref{de-100}) gives an embedding of the whole global representation 
$Ind_{P_2(\A)}^{Sp_4(\A)}(\chi\nu^{s}_{GL_2(\A)})$ in the space of automorphic forms.
\item Assume $s=-\frac{1}{2}.$ Then, If $\chi_p^2\neq 1$ then all the local intertwining operators are holomorphic isomorphisms(for all $p\le \infty$) and the image is isomorphic to $\chi_p\nu_p^{-1/2}1_{GL_2(\Q_p)}\rtimes 1.$  If $\chi_p^2=1$ but $\chi_p\neq 1$ all the intertwinig operators are homomorphisms. If $\chi_p=1$  the local intertwining operator  $\cal N(\Lambda_{s, p}, \wit{c_2sc_2})$ where $p\le \infty$ can have a  pole for for every $p$  with the choice of $f_p$ from the non-spherical subqotient. This gives us a pole of the Eisenstein series (for  $\chi$ which allows this possibility, e.g., $\chi=1$) of every possible order. In this case, the image (by (\ref{de-100}) after the normalization) consists of the irreducible tempered representations (on these ``problematic'' places where $\chi_p=1$ and a non-spherical (tempered) choice of a subquotient and of the representation generated by the spherical subquotient on the rest of the places (and the spherical subquotient generates the representation $\nu_p^{1/2}1_{GL_2(\Q_p)}\rtimes 1$).
\item Assume $-\frac{3}{2}<s<-\frac{1}{2}.$  Then either the Eisenstein series are holomorphic or they have a possible pole of the first order due to the pole of  $r(\Lambda_s,c_2)^{-1},\;r(\Lambda_s,sc_2)^{-1}$ and  $r(\Lambda_s,c_2sc_2)^{-1}$ coming from the zero of  $L(2s+\frac{3}{2},\chi)$ for given $s;$  in both cases the intertwining operators are holomorphic isomorphisms so that  (\ref{de-100}) gives an embedding of the whole global representation 
$Ind_{P_2(\A)}^{Sp_4(\A)}(\chi\nu^{s}_{GL_2(\A)})$ in the space of automorphic forms.

\item Assume $s=-\frac{3}{2}.$ Then, If $\chi_p^2\neq 1$ then all the local intertwining operators are holomorphic isomorphisms(for all $p\le \infty$) and the image is isomorphic to $\chi_p\nu_p^{-3/2}1_{GL_2(\Q_p)}\rtimes 1.$  Assume $p<\infty.$  If $\chi_p=1$  all the nontrivial  local intertwining operators  have a  pole for for every $p$  with the choice of $f_p$ from the non-spherical subqotient. On the spherical subrepresentation of $\chi_p\nu_p^{-3/2}1_{GL_2(\Q_p)}\rtimes 1$ (in this case just the trivial representation) all these operators are holomorphic, and the image generated in that way is isomorphic to the representation $\nu_p^{3/2}1_{GL_2(\Q_p)}\rtimes 1$ (thus reducible). For the choice of $f_p$ which does not belong to the spherical subrepresentation, all these operators have  a pole, and, after removing the pole by normalization, the image spans an irreducible subrepresentation $L(\nu_p^2;St_{SL_2(\Q_p)}).$  If $\chi_p\neq 1$ all the local intertwining operators are isomorphisms. At the archimedean place, all the choices of $f_p$ different from $L(\nu_{\infty}^2, \nu_{\infty}^1;1)$ if $\chi_{\infty}=1$ or $L(sgn\nu_{\infty}^2, sgn\nu_{\infty}^1;1)$ if $\chi_{\infty}=sgn$ give us poles of the intertwining operatprs.  This gives us a pole of the Eisenstein series (for  $\chi$ which allows this possibility, e.g., $\chi=1$) of every possible order.

\item Assume $s<-\frac{3}{2}.$ Then, if $\chi_{\infty}^2\neq 1$ or $s\pm \frac{1}{2}$ is not an integer, the Eisenstein series is holomorphic and  (\ref{de-100}) gives an embedding of the whole global representation 
$Ind_{P_2(\A)}^{Sp_4(\A)}(\chi\nu^{s}_{GL_2(\A)})$ in the space of automorphic forms. Assume that $\chi_{\infty}^2= 1$ and $s\pm \frac{1}{2}$ is an integer. Then, the Eisenstein series is holomorphic on $\otimes f_p,$ if $f_{\infty}$ belongs to the Langlands quotient $L(\chi_{\infty}\nu_{\infty}^{-(s-\frac{1}{2})},\chi_{\infty} \nu_{\infty}^{-(s+\frac{1}{2})};1)$ (which is a subrepresentation of  $\chi_{\infty}\nu^{s}1_{GL_2(\R)}\rtimes 1$). In this case, (\ref{de-100}) gives an embedding of an irreducible global representation $\otimes_{p<\infty} (\chi_{p}\nu^{s}1_{GL_2(\Q_p)}\rtimes 1)\otimes L(\chi_{\infty}\nu_{\infty}^{-(s-\frac{1}{2})},\chi_{\infty} \nu_{\infty}^{-(s+\frac{1}{2})};1)$ in the space of automorphic forms. If $f_{\infty}$ is not chosen in that way, the Eisenstein series have a pole of the first order coming from the pole of the archimedean intertwining operators (in the constant term). After the normalization, the image of (\ref{de-100}) in that case is a realization of a representation which has, on the archimedean place,  the unique maximal proper subrepresentation  of    $\nu_{\infty}^{-s}1_{GL_2(\R)}\rtimes 1.$
\end{enumerate}
\end{thm}
\begin{proof}
  The case of $-\frac{1}{2}<s<0$ is obvious. Assume $s=-\frac{1}{2}.$ Then, all the inverses of the global normalizing factors are holomorphic. For $p\le \infty,$ all the intertwining operators are holomorphic if $\chi_p^2\neq 1.$ Now assume $p<\infty.$ Assume that $\chi_p=1.$ From the discussion in Theorem \ref{thm:image_Siegel}, we know that we have
\[L(\nu_p^1;\nu_p^0\rtimes 1)\hookrightarrow \nu^{-1/2}1_{GL_2(\Q_p)}\rtimes 1 \to T_2,\]
where $T_2$ is a tempered representation, and $L(\nu_p^1;\nu_p^0\rtimes 1)$ the spherical subquotient. The operator $\cal N(\Lambda_{s, p}, \wit{c_2})$ is a holomorphic isomorphism on $\nu^{-1/2}1_{GL_2(\Q_p)}\rtimes 1$ (actually, an identity operator).The operator   $\cal N(\Lambda_{s, p}, \wit{s})$ acting on $\nu^{-1/2}1_{GL_2(\Q_p)}\rtimes 1$ (i.e.,$\cal N(\Lambda_{s, p}, \wit{sc_2})$ acting on $\nu^{-1/2}1_{GL_2(\Q_p)}\rtimes 1$ because of the previous remark) has a pole on the quotient $\nu^{-1/2}St_{GL_2(\Q_p)}\rtimes 1,$ but is holomorphic on  $\nu^{-1/2}1_{GL_2(\Q_p)}\rtimes 1.$ The operator  $\cal N(\Lambda_{s, p}, \wit{c_2})$ acting on $\nu^0\times \nu^{-1}\rtimes 1$ has  a pole on $\nu^{0}\rtimes SL_2(\Q_p)=T_1\otimes T_2,$ where $T_1$ is the other tempered representation. We conclude that if we pick $f_p$ from  $L(\nu_p^1;\nu_p^0\rtimes 1)$ then all the operators will be holomorphic $f_p$ (belongs to the spherical subquotient, generated by the spherical vector), but if we pick $f_p$ which belongs to the quotient $T_2,$ the operator  $\cal N(\Lambda_{s, p}, \wit{c_2sc_2})$ has a pole on that $f_p.$ 
Assume now that $\chi_p^2=1,$ but $\chi_p\neq 1.$ Then, we know by Theorem \ref{thm:image_Siegel} that the length of the representation  $\chi_p\nu^{-1/2}1_{GL_2(\Q_p)}\rtimes 1$ is three, and that it has two subrepresentations. Now, the operator $\cal N(\Lambda_{s, p}, \wit{c_2})$ acts  on $\chi_p\nu_p^{-1}\rtimes T_1$ as the identity, and on $\chi_p\nu_p^{-1}\rtimes T_2$ as minus identity. The operator  $\cal N(\Lambda_{s, p}, \wit{s})$ acting on $\chi_p\nu_p^{-1}\times \chi_p\nu_p^{0}\rtimes 1$ has a pole on $\chi_p\nu^{-1/2}St_{GL_2(\Q_p)}\rtimes 1,$ thus it does not have a pole on   $\chi_p\nu^{-1/2}1_{GL_2(\Q_p)}\rtimes 1.$
The operator  $\cal N(\Lambda_{s, p}, \wit{c_2})$ acts on  $\chi_p\nu_p^{0}\times \chi_p\nu_p^{-1}\rtimes 1$ as a holomorphic isomorphism. We conclude that all the operators are holomorphic on  $\chi_p\nu^{-1/2}1_{GL_2(\Q_p)}\rtimes 1.$
Now, assume that $p=\infty.$  Then, again  $\cal N(\Lambda_{s, \infty}, \wit{c_2})$ is holomorphic isomorphism, and  $\cal N(\Lambda_{s, \infty}, \wit{sc_2})$ does not have a pole on $\chi_{\infty}\nu^{-1/2}1_{GL_2(\R)}\rtimes 1.$  The operator  $\cal N(\Lambda_{s, \infty}, \wit{c_2})$ acting on $\chi_{\infty}\nu_{\infty}^0\times \chi_{\infty}^{-1}\nu_{\infty}^{-1}\rtimes 1$ can only have a pole
if $\chi_{\infty}^{-1}\nu_{\infty}^{-1}\rtimes 1$ is reducible, and that is if $\chi_{\infty}=1.$ In that case, we have
\[V_1=L(\nu_{\infty}^1;1) \hookrightarrow \nu_{\infty}\rtimes 1\to X(1,+)\oplus X(1,-),\]
where $X(1,\pm)$ are discrete series of $SL_2(\R)$ (cf. \cite{MuicSp4} Theorem 2.4).  The operator $\cal N(\Lambda_{s, p}, \wit{c_2})$
has a pole on the quotient $\nu_{\infty}^0\rtimes ( X(1,+)\oplus X(1,-)).$ This quotient decomposes as a sum of four temepred representations (cf. \cite{MuicSp4} Theorem 10.7), two of which are quotients of  $\nu_{\infty}^{-1/2}1_{GL_2(\R)}\rtimes 1.$ So, we conclude that the operator $\cal N(\Lambda_{s, \infty}, \wit{c_2sc_2})$ is holomorphic on the spherical subrepresentation of  $\nu_{\infty}^{-1/2}1_{GL_2(\R)}\rtimes 1$ (this also follows from the general properties of the normalized intertwining operators), but other choices of $f_p$ will gives us a pole.

Now assume $s=-\frac{3}{2}.$ Assume $p<\infty$ and $\chi_p=1$ (this is the only case when poles might occur). Then, $\cal N(\Lambda_{s, p}, \wit{c_2})$ has a pole on the quotient $\nu_p^{-2}\rtimes St_{SL_2(\Q_p)},$ which decomposes (in the appropriate Grothendick group) as $St_{Sp_4(\Q_p)}+L(\nu_p^2;St_{SL_2(\Q_p)}).$ On the other hand, we have
\[L(\nu_p^2,\nu_p^1;1)\hookrightarrow \nu_p^{-3/2}1_{GL_2(\Q_p)}\rtimes 1\to L(\nu_p^2;St_{SL_2(\Q_p)}).\]
Note that $ \cal N(\Lambda_{s, p}, \wit{s})$ acting on $\nu_p^{-2}\times \nu_p^1\rtimes 1$ is a holomorphic isomorphism, and
so is  $ \cal N(\Lambda_{s, p}, \wit{c_2})$ acting on $\nu_p^{1}\times \nu_p^{-2}\rtimes 1.$ We conclude that the images of  $ \cal N(\Lambda_{s, p}, \wit{c_2}),\;  \cal N(\Lambda_{s, p}, \wit{sc_2})$ and $ \cal N(\Lambda_{s, p}, \wit{c_2sc_2})$ will be mutually isomorphic on  $\nu_p^{-3/2}1_{GL_2(\Q_p)}\rtimes 1.$ We see that on the spherical subrepresentation (in this case just the trivial representation) all these operators are holomorphic, and the image generated in that way is isomorphic to the representation $\nu_p^{3/2}1_{GL_2(\Q_p)}\rtimes 1$ (thus reducible). If we pick $f_p$ which does not belong to the spherical 
$L(\nu_p^2,\nu_p^1;1)$, all these operators have  a pole, and, after removing the pole by normalization, the image spans an irreducible subrepresentation $L(\nu_p^2;St_{SL_2(\Q_p)}).$ If $\chi_p\neq 1$ all the local intertwining operators are isomorphisms. Now assume $p=\infty.$ Then, by Lemma \ref{lem:operatorsSiegel}, we see that the poles may occur if $\chi_{\infty}=1$ or $\chi_{\infty}=sgn.$ We first deal with the case $\chi_{\infty}=1.$ Then, already the operator $\cal N(\Lambda_{s, \infty}, \wit{c_2})$ has a pole on each the subquotients of  $\nu_{\infty}^{-3/2}1_{GL_2(\R)}\rtimes 1$, except on it's unique subrepresentation (which is again the trivial representation, as in the non-archimedean case), cf. \cite{MuicSp4} Theorem 11.1. Note that the action of the other intertwining operators cannot cancel this pole.  
Assume now that $\chi_{\infty}=sgn.$ Then,  $\cal N(\Lambda_{s, \infty}, \wit{c_2})$ acting on $sgn\nu_{\infty}^{-3/2}1_{GL_2(\R)}\rtimes 1$ is a holomorphic isomorphism. The pole of $\cal N(\Lambda_{s, \infty}, \wit{s})$ acting on $sgn\nu_{\infty}^{-2}\times sgn\nu_{\infty}^1\rtimes 1$ happens on the subquotients of $\delta(sgn\nu_{infty}^{1/2},3)\rtimes 1$ and $sgn\nu_{\infty}^{-3/2}1_{GL_2(\R)}\rtimes 1$ and this representation do not have common irreducible subquotients by \cite{MuicSp4}, proof of Theorem 10.1 and (10.64) there.  The pole of $\cal N(\Lambda_{s, \infty}, \wit{c_2})$ acting on $sgn\nu_{\infty}^{1}\times sgn\nu_{\infty}^{-2}\rtimes 1$ happens on the subquotients of $sgn\nu_{\infty}^1\rtimes X(2,+)\oplus sgn\nu_{\infty}^1\rtimes X(2,-).$ But the representation
n $sgn\nu_{\infty}^{-3/2}1_{GL_2(\R)}\rtimes 1$ is of length three and two of its subquotients (so the ones different from $L(sgn\nu_{\infty}^{2}, sgn\nu_{\infty}^1;1)$) appear in this space where the poles occur. Note that the representation $sgn\nu_{\infty}^{2}\times sgn\nu_{\infty}^1\rtimes 1$ is multiplicity free.

From Lemma \ref{lem:operatorsSiegel} we now that in the case $s<-\frac{3}{2}$ it is enough to see what is happening on the archimedean places (all the local intertwining operators at the non-archimedean places are holomorphic isomorphisms) the following situations:
\begin{enumerate}
\item  $\chi_{\infty}=1,$ $s+\frac{1}{2}$ is odd integer,
\item  $\chi_{\infty}=1,$ $s+\frac{1}{2}$ is even integer,
\item $\chi_{\infty}=sgn,$ $s+\frac{1}{2}$ is an odd integer,
\item  $\chi_{\infty}=sgn,$ $s+\frac{1}{2}$ is an even integer.

\end{enumerate}
Assume the first possibility. Then from Theorem 11.1 of \cite{MuicSp4}, we see that  $\cal N(\Lambda_{s, \infty}, \wit{c_2})$ acting on 
$\nu_{\infty}^{s}1_{GL_2(\R)}\rtimes 1$ has a pole of the first order for every choice $f_p$ which does not belong to the Langlands quotient $L(\nu_{\infty}^{-s+1/2},\nu_{\infty}^{-s-1/2};1)$ (the representation $sgn\nu_{\infty}^{s}1_{GL_2(\R)}\rtimes 1$ is of the length four). The second intertwining operator has a pole on $\nu_{\infty}^{s-1/2}\times \nu_{\infty}^{-(s+1/2)}\rtimes 1,$ but it does not add a pole of the higher order on the image of $\cal N(\Lambda_{s, \infty}, \wit{c_2})(\nu_{\infty}^{s}1_{GL_2(\R)}\rtimes 1)$ (and is non-zero on the subquotients of $\nu_{\infty}^{s}1_{GL_2(\R)}\rtimes 1$).  The intertwining operator  $\cal N(\Lambda_{s, \infty}, \wit{c_2})$ acting on  $\nu_{\infty}^{-(s+1/2)}\times \nu_{\infty}^{s-1/2}\rtimes 1$ is a holomorphic isomorphism by Lemma \ref{lem:operatorsSiegel}.
In the second possibility, the first intertwining operator  $\cal N(\Lambda_{s, \infty}, \wit{c_2})$  acting on $\nu_{\infty}^{s}1_{GL_2(\R)}\rtimes 1$  is a holomorphic isomorphism. The second intertwining operator has a pole on $\delta(-1/2,2s)\rtimes 1$ (again notation from \cite{MuicSp4},  the proof of Theorem 10.1) which is a representation of the length three and does not have a common subquotient with the representation  $\nu_{\infty}^{s}1_{GL_2(\R)}\rtimes 1$  (cf. Theorem 10.6 of \cite{MuicSp4}), so it's holomorphic and non-zero on it.
The third intertwining operator   $\cal N(\Lambda_{s, \infty}, \wit{c_2})$ acts on $\nu_{\infty}^{-(s+1/2)}\times \nu_{\infty}^{s-1/2}\times 1,$ and pole is obtained on the subquotients of  $\nu_{\infty}^{-(s+1/2)}\rtimes (X(-(s-\frac{1}{2}),+)\oplus X(-(s-\frac{1}{2}),-).$ All the subquotients of  $\nu_{\infty}^{s}1_{GL_2(\R)}\rtimes 1,$ except $L(\nu_{\infty}^{-(s-\frac{1}{2})}, \nu_{\infty}^{-(s+\frac{1}{2})};1)$ are among those on which the pole occurs. 

The situation with the rest of the cases is totally symmetric (i.e., one uses Theorem 10.1, Theorem 10.6 or Theorem 11.1 of\cite{MuicSp4}), and we conclude that in all of these four cases, at least one of the operators $\cal N(\Lambda_{s, \infty}, \wit{c_2}),\; \cal N(\Lambda_{s, \infty}, \wit{sc_2})$ and $ \cal N(\Lambda_{s, \infty}, \wit{c_2sc_2})$ has a pole of the first order on each of the subquotients of $sgn^{\varepsilon}\nu_{\infty}^{s}1_{GL_2(\R)}\rtimes 1,$ except $L(sgn^{\varepsilon}\nu_{\infty}^{-(s-\frac{1}{2})}, sgn^{\varepsilon}\nu_{\infty}^{-(s+\frac{1}{2})};1)$ (here $\varepsilon\in \{0,1\}$). If one removes the pole, the image spans the unique maximal proper subrepresentation  of    $\nu_{\infty}^{-s}1_{GL_2(\R)}\rtimes 1.$ 
\end{proof}

\bibliographystyle{siam}
\bibliography{degenerate_eisenstein}
\bigskip
\end{document}